\documentclass[twoside,10pt]{article}
\usepackage{indentfirst}                              
\usepackage[colorlinks]{hyperref}
\hypersetup{linkcolor=black,filecolor=black,urlcolor=black, citecolor=black}  
\usepackage[numbers,sort&compress]{natbib} 
\usepackage[perpage,symbol*]{footmisc}  

\usepackage{amsmath} 
\usepackage{amsfonts}
\usepackage{amssymb} 
\usepackage{bm}           
\usepackage{mathrsfs} 
\usepackage{amsthm}  
\usepackage{accents}  

\usepackage[titletoc,title]{appendix}
\usepackage{titlesec}
\titleformat{\section}{\large\bfseries}{\thesection}{1em}{}
\titleformat{\subsection}[runin]{\bfseries}{\thesubsection.}{0.5em}{}[.]
\titleformat{\subsubsection}[runin]{\bfseries}{\thesubsubsection.}{0.4em}{}[.]

\DeclareMathOperator{\supp}{supp}

\DeclareMathOperator{\sgn}{sgn}

\def\d{\,\mathrm{d}}
\def\p{\partial}

\makeatletter
\def\wideubar{\underaccent{{\cc@style\underline{\mskip10mu}}}}
\def\Wideubar{\underaccent{{\cc@style\underline{\mskip8mu}}}}
\makeatother
\makeatletter
\def\widebar{\accentset{{\cc@style\underline{\mskip10mu}}}}
\def\Widebar{\accentset{{\cc@style\underline{\mskip8mu}}}}
\makeatother

\newcommand{\VERTiii}[1]{{\left\vert\kern-0.3ex\left\vert\kern-0.3ex\left\vert #1
 \right\vert\kern-0.3ex\right\vert\kern-0.3ex\right\vert}}
\newcommand{\VERT}{\vert\kern-0.3ex\vert\kern-0.3ex\vert}
\newcommand{\VERTl}{\left\vert\kern-0.3ex\left\vert\kern-0.3ex\left\vert}
\newcommand{\VERTr}{\right\vert\kern-0.3ex\right\vert\kern-0.3ex\right\vert}
\newcommand{\VERTbig}{\big\vert\kern-0.3ex\big\vert\kern-0.3ex\big\vert}
\newcommand{\VERTBig}{\Big\vert\kern-0.3ex\Big\vert\kern-0.3ex\Big\vert}

\makeatletter
\DeclareFontFamily{OMX}{MnSymbolE}{}
\DeclareSymbolFont{MnLargeSymbols}{OMX}{MnSymbolE}{m}{n}
\SetSymbolFont{MnLargeSymbols}{bold}{OMX}{MnSymbolE}{b}{n}
\DeclareFontShape{OMX}{MnSymbolE}{m}{n}{
 <-6>  MnSymbolE5 <6-7>  MnSymbolE6 <7-8>  MnSymbolE7 <8-9>  MnSymbolE8 <9-10> MnSymbolE9 <10-12> MnSymbolE10 <12->   MnSymbolE12
}{}
\DeclareFontShape{OMX}{MnSymbolE}{b}{n}{
 <-6>  MnSymbolE-Bold5 <6-7>  MnSymbolE-Bold6 <7-8>  MnSymbolE-Bold7 <8-9>  MnSymbolE-Bold8 <9-10> MnSymbolE-Bold9 <10-12> MnSymbolE-Bold10 <12->   MnSymbolE-Bold12
}{}

\let\llangle\@undefined
\let\rrangle\@undefined
\DeclareMathDelimiter{\llangle}{\mathopen}%
{MnLargeSymbols}{'164}{MnLargeSymbols}{'164}
\DeclareMathDelimiter{\rrangle}{\mathclose}%
{MnLargeSymbols}{'171}{MnLargeSymbols}{'171}
\makeatother

\numberwithin{equation}{section}

 \newtheorem{lemma}{Lemma}[section]
 
 \newtheorem{theorem}{Theorem}[section]
 \newtheorem{corollary}[lemma]{Corollary}

 \newtheorem{remark}{Remark}[section]

\newcommand{\R}{{\mathbb R}}
\newcommand{\C}{{\mathbb C}}

\def\ds{\displaystyle}

\def\p{\partial}
\def\duno{\partial_1}
\def\dquno{\partial^2_{11} }
\def\ddue{\partial_2}

\def\di{\partial_i}
\def\dj{\partial_j}

\def\dt{\partial_t}
\def\dqt{\partial^2_{tt}}

\def\F{\mathcal F}

\def\Lc{\mathcal L}
\def\Rc{\mathcal R}
\def\V{\mathcal V}

\def\vv{\bf v}
\def\ww{\bf w}
\def\VV{\bf V}

\begin{document}
\title{{\bf On the evolution equation of compressible\\ vortex sheets}}
\author{
{\sc Alessandro Morando}\thanks{e-mail: alessandro.morando@unibs.it}\;,
{\sc Paolo Secchi}\thanks{e-mail: paolo.secchi@unibs.it}\;,
{\sc Paola Trebeschi}\thanks{e-mail: paola.trebeschi@unibs.it}\\
{\footnotesize DICATAM, Sezione di Matematica,
Universit\`a di Brescia, Via Valotti 9, 25133 Brescia, Italy}
}

\date{}

\maketitle
\begin{abstract}
We are concerned with supersonic vortex sheets for the Euler equations of compressible inviscid fluids in two space dimensions. For the problem with constant coefficients we derive an evolution equation for the discontinuity front of the vortex sheet. This is a pseudo-differential equation of order two. In agreement with the classical stability analysis, if the jump of the tangential component of the velocity satisfies $|[v\cdot\tau]|<2\sqrt{2}\,c$ (here $c$ denotes the sound speed) the symbol is elliptic and the problem is ill-posed. On the contrary, if $|[v\cdot\tau]|>2\sqrt{2}\,c$, then the problem is weakly stable, and we are able to derive a wave-type a priori energy estimate for the solution, with no loss of regularity with respect to the data.
Then we prove the well-posedness of the problem, by showing the existence of the solution in  weighted Sobolev spaces.

\vspace{2mm}

\vspace{2mm}

\end{abstract}

\tableofcontents

\section{Introduction}
\label{sect1}

We are concerned with the time evolution of vortex sheets for the Euler equations describing the motion of a compressible fluid.
Vortex sheets are interfaces between two incompressible or compressible flows across which there is a discontinuity in fluid velocity. Across a vortex sheet, the tangential velocity field has a jump, while the normal component of the flow velocity is continuous. The discontinuity in the tangential velocity field creates a concentration of vorticity along the interface. In particular, compressible vortex sheets are contact discontinuities to the Euler equations for compressible fluids and as such they are fundamental waves which play an important role in the study of general entropy solutions to multidimensional hyperbolic systems of conservation laws.

It was observed in \cite{M58MR0097930,FM63MR0154509}, by the normal mode analysis, that rectilinear vortex sheets for isentropic compressible fluids in two space dimensions are linearly stable when the Mach number $\mathsf{M}>\sqrt{2}$ and are violently unstable when $\mathsf{M}<\sqrt{2}$, while planar vortex sheets are always violently unstable in three space dimensions. This kind of instabilities is the analogue of the Kelvin--Helmholtz instability for incompressible fluids.
%
%
 \citet{AM87MR914450} studied certain instabilities of two-dimensional supersonic vortex sheets by analyzing the interaction with highly oscillatory waves through geometric optics. A rigorous mathematical theory on nonlinear stability and local-in-time existence of two-dimensional supersonic vortex sheets was first established by Coulombel--Secchi \cite{CS08MR2423311,CS09MR2505379} based on their linear stability results in \cite{CS04MR2095445} and a Nash--Moser iteration scheme.

Characteristic discontinuities, especially vortex sheets, arise in a broad range of physical problems in fluid mechanics, oceanography, aerodynamics, plasma physics, astrophysics, and elastodynamics. The linear results in \cite{CS04MR2095445} have been generalized to cover the two-dimensional nonisentropic flows \cite{MT08MR2441089}, the three-dimensional compressible steady flows \cite{WY13MR3065290,WYuan15MR3327369}, and the two-dimensional two-phase flows \cite{RWWZ16MR3474128}.
Recently, the methodology in \cite{CS04MR2095445} has been developed to deal with several constant coefficient linearized problems arising in two-dimensional compressible magnetohydrodynamics (MHD) and elastic flows; see \cite{WY13ARMAMR3035981,CDS16MR3527627,CHW17Adv}. For three-dimensional MHD, Chen--Wang \cite{CW08MR2372810} and \citet{T09MR2481071} proved the nonlinear stability of compressible current-vortex sheets, which indicates that non-paralleled magnetic fields stabilize the motion of three-dimensional compressible vortex sheets. Moreover, the modified Nash--Moser iteration scheme developed in \cite{H76MR0602181,CS08MR2423311} has been successfully applied to the compressible liquids in vacuum \cite{T09MR2560044}, the plasma-vacuum interface problem \cite{ST14MR3151094}, three-dimensional compressible steady flows \cite{WY15MR3328144}, and MHD contact discontinuities \cite{MTT16Preprint}.
The approach of \cite{CS04MR2095445, CS08MR2423311} has been recently extended to get the existence of solutions to the non linear problem of relativistic vortex sheets in three-dimensional Minkowski spacetime \cite{CSW1707.02672}, and the two-dimensional nonisentropic flows \cite{MTW17MR}.

The vortex sheet motion is a nonlinear hyperbolic problem with a characteristic free boundary.
The analysis of the linearized problem in \cite{CS04MR2095445} shows that the so-called Kreiss-Lopatinski\u{\i} condition holds in a weak sense, thus one can only obtain an \emph{a priori} energy estimate with a loss of derivatives with respect to the source terms. Because of this fact, the existence of the solution to the nonlinear problem is obtained in \cite{CS08MR2423311} by a Nash-Moser iteration scheme, with a loss of the regularity of the solution with respect to the initial data.

At the best of our knowledge the approach of \cite{CS04MR2095445,CS08MR2423311} is the only one known up to now, while it would be interesting to have different methods of proof capable to give the existence and possibly other properties of the solution.



In particular, the location of the discontinuity front of the vortex sheet is obtained through the jump conditions at the front, see \eqref{RH}, and is implicitly determined by the fluid motion in the interior regions, i.e. far from the front.
On the contrary, it would be interesting to find an \lq\lq{explicit}\rq\rq evolution equation for the vortex sheet, i.e. for the discontinuity front, that might also be useful for numerical simulations. In this regard we recall that in case of irrotational, incompressible vortex sheets, the location of the discontinuity front is described by the Birchhoff-Rott equation, see \cite{MR1688875,MB02MR1867882,MP94MR1245492}, whose solution is sufficient to give a complete description of the fluid motion through the Biot-Savart law.
The evolution equation of the discontinuity front of current-vortex sheets plays an important role in the paper \cite{SWZ}.

In this paper we are concerned with supersonic vortex sheets for the Euler equations of compressible inviscid fluids in two space dimensions. For the problem with constant coefficients we are able to derive an evolution equation for the discontinuity front of the vortex sheet. This is a pseudo-differential equation of order two. In agreement with the classical stability analysis \cite{FM63MR0154509,M58MR0097930}, if the jump of the tangential component of the velocity satisfies $|[v\cdot\tau]|<2\sqrt{2}\,c$ (here $c$ denotes the sound speed) the symbol is elliptic and the problem is ill-posed. On the contrary, if $|[v\cdot\tau]|>2\sqrt{2}\,c$, then the problem is weakly stable, and we are able to derive a wave-type a priori energy estimate for the solution, with no loss of regularity with respect to the data.
By a duality argument we then prove the well-posedness of the problem, by showing the existence of the solution in weighted Sobolev spaces.

The fact that the evolution equation for the discontinuity front is well-posed, with no loss of regularity from the data to the solution, is somehow in agreement with the result of the linear analysis in \cite{CS04MR2095445} (see Theorem 3.1 and Theorem 5.2), where the solution has a loss of derivatives in the interior domains while the function describing the front conserves the regularity of the boundary data.

In a forthcoming paper we will consider the problem with variable coefficients, which requires a completely different approach.

\subsection{The Eulerian description}

We consider the isentropic Euler equations in the whole plane $\R^2$. Denoting by ${\bf v}=(v_1,v_2) \in \R^2$ the
velocity of the fluid, and by $\rho$ its density, the equations read:
\begin{equation}
\label{euler}
\begin{cases}
\dt \rho +\nabla \cdot (\rho \, {\bf v}) =0 \, ,\\
\dt (\rho \, {\bf v}) +\nabla \cdot
(\rho \, {\bf v} \otimes {\bf v}) +\nabla \, p =0 \, ,
\end{cases}
\end{equation}
where $p=p(\rho)$ is the pressure law. In all this paper $p$ is a $C^\infty$ function of $\rho$,
defined on $]0,+\infty[$, and such that $p'(\rho)>0$ for all $\rho$. The speed of sound $c(\rho)$ in the
fluid is defined by the relation:
\begin{equation*}
\forall \, \rho>0 \, ,\quad c(\rho) :=\sqrt{p'(\rho)} \, .
\end{equation*}
It is a well-known fact that, for such a pressure law, \eqref{euler} is a strictly hyperbolic system in
the region $(t,x)\in\, ]0,+\infty[ \, \times \R^2$, and \eqref{euler} is also symmetrizable.

We are interested in solutions of \eqref{euler} that are smooth on either side of a smooth hypersurface  $\Gamma(t):=\{x=(x_1,x_2)\in \R^2 : F(t,x)=0\}=\{x_2=f(t,x_1)\}$ for each $t$ and that satisfy
suitable jump conditions at each point of the front $\Gamma (t)$.

Let us denote $\Omega^\pm(t):=\{(x_1,x_2)\in \R^2 :x_2\gtrless f(t,x_1)\}$. Given any
function $g$ we denote $g^\pm=g$ in $\Omega^\pm(t)$ and $[g]=g^+_{|\Gamma}-g^-_{|\Gamma}$ the jump across
$\Gamma (t)$.

We look for smooth solutions $({\vv}^\pm,\rho^\pm)$ of \eqref{euler} in $\Omega^\pm(t)$ and such that, at each time $t$,
the tangential velocity is the only quantity that experiments a jump across the curve $\Gamma (t)$. (Tangential
should be understood as tangential with respect to $\Gamma (t)$). The pressure and the normal velocity should be
continuous across $\Gamma (t)$. For such solutions, the jump conditions across $\Gamma(t)$ read:
\begin{equation*}
\sigma ={\vv}^\pm\cdot n  \, ,\quad [p]=0 \quad {\rm on } \;\Gamma (t) \, .
\end{equation*}
Here $n=n(t)$ denotes the outward unit normal on $\partial\Omega^-(t)$ and $\sigma$ denotes the velocity of
propagation of the interface $\Gamma (t)$. With our parametrization of $\Gamma (t)$, an equivalent formulation
of these jump conditions is
\begin{equation}
\label{RH}
\dt f ={\bf v}^+\cdot N ={\bf v}^-\cdot N \, ,\quad p^+ =p^-  \quad {\rm on }\;\Gamma (t) \, ,
\end{equation}
where
\begin{equation}\label{defN}
N=(-\duno f, 1)
\end{equation}
and $p^\pm=p(\rho^\pm)$. Notice that the function $f$ describing the discontinuity front is part of
the unknown of the problem, i.e. this is a free boundary problem.
%
%
%

For smooth solutions system \eqref{euler} can be written in the equivalent form
\begin{equation}
\label{euler1}
\begin{cases}
\dt \rho +({\bf v}\cdot\nabla) \rho +\rho \, \nabla\cdot{\bf v} =0 \, ,\\
\rho \,(\dt  {\bf v} +({\bf v}\cdot\nabla)
 {\bf v} ) +\nabla \, p =0 \, .
\end{cases}
\end{equation}
Because $ p'(\rho)>0 $, the function $p= p(\rho ) $ can be inverted and we can write $ \rho=\rho(p) $. Given a positive constant $ \bar{\rho}>0 $, we introduce the quantity $ P(p)=\log(\rho(p)/\bar{\rho}) $ and consider $ P $ as a new unknown. In terms of $ (P,\vv) $, the system \eqref{euler1} equivalently reads
\begin{equation}
\label{euler2}
\begin{cases}
\dt P +{\bf v}\cdot\nabla P + \nabla\cdot{\bf v} =0 \, ,\\
\dt  {\bf v} +({\bf v}\cdot\nabla)
{\bf v}  +c^2\,\nabla \, P =0 \, ,
\end{cases}
\end{equation}
where now the speed of sound is considered as a function of $ P,$ that is $ c=c(P) $. Thus our problem reads
\begin{equation}
\label{euler3}
\begin{cases}
\dt P^\pm +{\bf v}^\pm\cdot\nabla P^\pm + \nabla\cdot{\bf v}^\pm =0 \, ,\\
\dt  {\bf v}^\pm +({\bf v}^\pm\cdot\nabla)
{\bf v}^\pm  +c^2_\pm\,\nabla \, P^\pm =0 \, , \qquad {\rm in }\;  \Omega^\pm(t),
\end{cases}
\end{equation}
where we have set $ c_\pm=c(P^\pm) $.
 The jump conditions \eqref{RH} 
 take the new form
\begin{equation}
\label{RH2}
\dt f ={\bf v}^+\cdot N ={\bf v}^-\cdot N \, ,\quad P^+ =P^-  \quad {\rm on }\;\Gamma (t) \, .
\end{equation}


\section{Preliminary results}

Given functions $ {\vv}^\pm, P^\pm$, we set
\begin{equation}
\begin{array}{ll}\label{defZ}
Z^\pm:=\dt {\vv}^\pm
+( {\vv}^\pm \cdot \nabla) {\vv}^\pm .
\end{array}
\end{equation}
Next, we study the behavior of $Z^\pm$  at $\Gamma(t)$. As in \cite{SWZ} we define
\begin{equation}
\begin{array}{ll}\label{deftheta}
\theta(t,x_1):= {\vv}^\pm(t,x_1,f(t,x_1))\cdot N(t,x_1),
\end{array}
\end{equation}
for $N$ given in \eqref{defN}.

\begin{lemma}\label{lemmaN}
Let $ f, {\vv}^\pm,  \theta$ be such that
\begin{equation}
\begin{array}{ll}\label{dtfthetavN}
\dt f=\theta= {\vv}^\pm\cdot N  \,  \qquad {\rm on }\;  \Gamma(t),
\end{array}
\end{equation}
and let $Z^\pm$ be defined by \eqref{defZ}.
Then
\begin{equation}
\begin{array}{ll}\label{applN}
Z^+ \cdot N
\ds = \dt\theta + 2 v_1^+\duno\theta + (v_1^+)^2 \dquno f \,,\\
Z^- \cdot N
\ds = \dt\theta + 2 v_1^-\duno\theta + (v_1^-)^2 \dquno f
 \quad {\rm on }\;  \Gamma(t) .
\end{array}
\end{equation}
\end{lemma}
\begin{proof}
Dropping for convenience the $\pm$ superscripts, we compute
\begin{equation*}
\begin{array}{ll}\label{}
\dt \theta=(\dt{\vv}+\ddue{\vv}\dt f)\cdot N + {\vv}\cdot \dt N=(\dt v_2+\ddue v_2\dt f) -(\dt v_1+\ddue v_1\dt f)\duno f - v_1 \dt \duno f \,,
\end{array}
\end{equation*}
and similarly
\begin{equation*}
\begin{array}{ll}\label{}
\duno \theta=(\duno v_2+\ddue v_2\duno f) -(\duno v_1+\ddue v_1\duno f)\duno f - v_1 \dquno f \, .
\end{array}
\end{equation*}
Substituting \eqref{dtfthetavN} in the first of the two equations it follows that
\begin{equation*}
\begin{array}{ll}\label{}
\dt v_2-\dt v_1\duno f=\dt \theta+v_1\duno \theta - \dt f(\ddue v_2 -\ddue v_1\duno f  ) \,,
\end{array}
\end{equation*}
and from the second equation, after multiplication by $v_1$, it follows that
\begin{equation*}
\begin{array}{ll}\label{}
v_1\duno v_2- v_1\duno v_1\duno f=v_1\duno \theta+v_1^2\dquno f - v_1\duno f(\ddue v_2 -\ddue v_1\duno f  ) \,.
\end{array}
\end{equation*}
We substitute the last two equations in
\begin{equation*}
\begin{array}{ll}\label{}
Z\cdot N=(\dt v_2
+ {\vv} \cdot \nabla v_2)-(\dt v_1
+ {\vv} \cdot \nabla v_1)\duno f \,,
\end{array}
\end{equation*}
rearrange the terms, use again \eqref{dtfthetavN}, and finally obtain
\[
Z \cdot N
\ds = \dt\theta + 2 v_1\duno\theta + v_1^2 \dquno f \,,
 \]
that is \eqref{applN}.
\end{proof}

\subsection{A first equation for the front}

We take the scalar product of the equation for $\vv^\pm$ in \eqref{euler3}, evaluated at $\Gamma(t)$, with the vector $N$. We get
\begin{equation*}
\big\{ Z^\pm  + c^2_\pm  \nabla P^\pm\big\} \cdot N =0\, \quad {\rm on} \; \Gamma(t) \, ,
\end{equation*}
and applying Lemma \ref{lemmaN} we obtain
\begin{equation}
\begin{array}{ll}\label{puntoN}
\ds  \dt\theta + 2 v_1^\pm\duno\theta + (v_1^\pm)^2 \dquno f + c^2_\pm  \nabla P^\pm \cdot N =0 \,\quad {\rm on} \; \Gamma(t) \, .
\end{array}
\end{equation}
Now we apply an idea from \cite{SWZ}. We take the {\it sum} of the "+" and "-" equations in \eqref{puntoN} to obtain
\begin{multline}\label{puntoN2}
\ds  2\dt\theta + 2 (v_1^++v_1^-)\duno\theta + ( (v_1^+)^2+(v_1^-)^2) \dquno f
+ c^2 \nabla (P^+ +  P^-)  \cdot N =0 \,\quad {\rm on} \; \Gamma(t) \, ,
\end{multline}
where we have denoted the common value at the boundary $c=c_{\pm|\Gamma(t)}=c(P^\pm_{|\Gamma(t)})$.
Next, following again \cite{SWZ}, we introduce the quantities
\begin{equation}
\label{defwV}
{\ww}=(w_1,w_2):=({\vv}^++{\vv}^-)/2, \qquad {\VV}=(V_1,V_2):=({\vv}^+-{\vv}^-)/2.
\end{equation}
Sustituting \eqref{defwV} in \eqref{puntoN2} gives
\begin{equation}\label{puntoN3}
\ds  \dt\theta + 2 w_1\duno\theta + (w_1^2 + V_1^2 )\dquno f
+\frac{c^2}2  \nabla (P^+ +  P^-)  \cdot N =0 \,\qquad {\rm on} \;  \Gamma(t) \, .
\end{equation}
Finally we substitute the boundary condition $ \theta=\dt f $ in \eqref{puntoN3} and we obtain
\begin{equation}\label{puntoN4}
\ds  \dqt f + 2 w_1\duno\dt f + (w_1^2 + V_1^2 )\dquno f
+\frac{c^2}2  \nabla (P^+ +  P^-)  \cdot N =0 \,\qquad {\rm on} \;  \Gamma(t) \, .
\end{equation}
\eqref{puntoN4} is a second order equation for the front $f$. However, it is nonlinearly coupled at the highest order with the other unknowns $ ({\vv}^\pm,P^\pm) $ of the problem through the last term in the left side of \eqref{puntoN4}. In order to find an evolution equation for $f,$ it is important to isolate the dependence of $f$ on $P^\pm$ at the highest order, i.e. up to lower order terms in $ ({\vv}^\pm,P^\pm) $.

Notice that \eqref{puntoN4} can also be written in the form
\begin{equation}\label{puntoN5}
\ds  (\dt  +  w_1\duno)^2 f +  V_1^2 \dquno f
+\frac{c^2}2  \nabla (P^+ +  P^-)  \cdot N -(\dt w_1+w_1\duno w_1)\duno f=0 \,\qquad {\rm on} \; \Gamma(t) \, .
\end{equation}

\subsection{The wave problem for the pressure}

Applying the operator $ \dt+{\vv}\cdot\nabla $ to the first equation of \eqref{euler2} and $ \nabla\cdot $ to the second one gives
\begin{equation*}\label{}
\begin{cases}
(\dt  +{\bf v}\cdot\nabla)^2 P + (\dt  +{\bf v}\cdot\nabla)\nabla\cdot{\bf v} =0 \, ,\\
\nabla\cdot(\dt  +{\bf v}\cdot\nabla)
{\bf v}  +\nabla\cdot(c^2\,\nabla \, P) =0 \, .
\end{cases}
\end{equation*}
The difference of the two equations gives the wave-type equation\footnote[1]{Here we adopt the Einstein convention over repeated indices.}
\begin{equation}\label{wave0}
(\dt  +{\bf v}\cdot\nabla)^2 P - \nabla\cdot(c^2\,\nabla \, P) = -[\dt  +{\bf v}\cdot\nabla, \nabla\cdot\,]{\vv}=\di v_j\dj v_i.
\end{equation}
We repeat the same calculation for both $ ({\vv}^\pm,P^\pm) $.
As for the behavior at the boundary, we already know that
\begin{equation}\label{bc1}
 [P]=0 \, , \qquad
{\rm on }\;  \Gamma(t) \, .
\end{equation}
As a second boundary condition it is natural to add a condition involving the normal derivatives of $ P^\pm. $ We proceed as follows: instead of the {\it sum} of the equations \eqref{puntoN} as for \eqref{puntoN2}, we take the {\it difference} of the "+" and "-" equations in \eqref{puntoN} to obtain the jump of the normal derivatives $ \nabla P^\pm  \cdot N $,
\begin{equation}
\label{jumpQ}
[c^2 \nabla P \cdot N] =-[2 v_1\duno\theta + v_1^2 \dquno f] \qquad {\rm on} \; \Gamma(t) \, .
\end{equation}
Recalling that $ \theta=\dt f $, we compute
\begin{equation}\label{jumpQ1}
[2 v_1\duno\theta + v_1^2 \dquno f] = 4 V_1(\dt+w_1\duno)\duno f .
\end{equation}
Thus, from \eqref{jumpQ}, \eqref{jumpQ1} we get
\begin{equation}\label{bc2}
[c^2 \nabla P \cdot N] =-4 V_1(\dt+w_1\duno)\duno f \qquad {\rm on} \;  \Gamma(t) \, .
\end{equation}
Collecting \eqref{wave0} for $P^\pm$, \eqref{bc1}, \eqref{bc2} gives the coupled problem for the pressure
\begin{equation}\label{wave}
\begin{cases}
(\dt  +{\vv }^\pm\cdot\nabla)^2 P^\pm - \nabla\cdot(c^2_\pm\,\nabla \, P^\pm) =\F^\pm  &
{\rm in }\;  \Omega^\pm(t) \, ,\\
 [P]=0 \, ,\\
[c^2 \nabla P \cdot N] =-4 V_1(\dt+w_1\duno)\duno f & {\rm on} \; \Gamma(t) \, ,
\end{cases}
\end{equation}
where
\begin{equation*}\label{key}
\F^\pm:=\di v_j^\pm \dj v_i^\pm.
\end{equation*}
Notice that $ \F $ can be considered a lower order term in the second order differential equation for $ P^\pm $, differently from the right-hand side of the boundary condition for the jump of the normal derivatives, which is of order two in $ f. $

\section{The coupled problem \eqref{puntoN5}, \eqref{wave} with constant coefficients. The main result}

We consider a problem obtained by linearization of equation \eqref{puntoN5} and system \eqref{wave} about the constant velocity ${\vv }^\pm=(v_1^\pm,0)$, constant pressure $P^+=P^-$, and flat front $\Gamma=\{x_2=0\}$, so that $N=(0,1)$, that is we study the equations
\begin{equation}\label{puntoN6}
\ds  (\dt  +  w_1\duno)^2 f +  V_1^2 \dquno f
+\frac{c^2}2 \ddue (P^+ +  P^-)   =0 \,\qquad {\rm if} \; x_2=0 \, ,
\end{equation}
\begin{equation}\label{wave2}
\begin{cases}
(\dt  +v_1^\pm\duno)^2 P^\pm - c^2\Delta \, P^\pm =\F^\pm  \quad &
{\rm if }\; x_2\gtrless0 \, ,\\
[P]=0 \, ,\\
[c^2\ddue P ] =-4 V_1(\dt+w_1\duno)\duno f & {\rm if} \; x_2=0 \, .
\end{cases}
\end{equation}
In \eqref{puntoN6}, \eqref{wave2}, $v^\pm_1, c$ are constants and $c>0$, $w_1=(v_1^++v_1^-)/2, V_1=(v_1^+-v_1^-)/2$. $\F^\pm$ is a given source term. \eqref{puntoN6}, \eqref{wave2} form a coupled system for $f$ and $P^\pm$, obtained by retaining the highest order terms of \eqref{puntoN5} and \eqref{wave}.  We are interested to derive from \eqref{puntoN6}, \eqref{wave2} an evolution equation for the front $f$.

For $\gamma\ge1$, we introduce $ \widetilde{f}:=e^{-\gamma t}f,\widetilde{P}^\pm:=e^{-\gamma t}P^\pm, \widetilde{\F}^\pm:=e^{-\gamma t}\F^\pm $ and consider the equations
\begin{equation}\label{puntoN7}
\ds  (\gamma+\dt +  w_1\duno)^2 \widetilde{f} +  V_1^2 \dquno \widetilde{f}
+\frac{c^2}2 \ddue (\widetilde{P}^+ +  \widetilde{P}^-)   =0 \,\qquad {\rm if} \; x_2=0 \, ,
\end{equation}
\begin{equation}\label{wave3}
\begin{cases}
(\gamma+\dt  +v_1^\pm\duno)^2 \widetilde{P}^\pm - c^2\Delta \, \widetilde{P}^\pm =\widetilde{\F}^\pm  \quad &
{\rm if }\; x_2\gtrless0 \, ,\\
[\widetilde{P}]=0 \, ,\\
[c^2\ddue \widetilde{P} ] =-4 V_1(\gamma+\dt+w_1\duno)\duno \widetilde{f} & {\rm if} \; x_2=0 \, .
\end{cases}
\end{equation}
System \eqref{puntoN7}, \eqref{wave3} is equivalent to \eqref{puntoN6}, \eqref{wave2}. Let us denote by $ \widehat{f},\widehat{P}^\pm,\widehat{\F}^\pm $ the Fourier transforms of $ \widetilde{f},\widetilde{P}^\pm, \widetilde{\F}^\pm $ in $(t,x_1)$, with dual variables denoted by $(\delta,\eta)$, and set $\tau=\gamma+i\delta$. We have the following result:
\begin{theorem}\label{teo_equ}
Let $\widetilde{\F}^\pm$ be such that
\begin{equation}\label{cond_infF}
\lim\limits_{x_2\to+\infty}\widehat{\F}^\pm(\cdot,\pm x_2)= 0 \, .
\end{equation}
Assume that $\widetilde{f},\widetilde{P}^\pm$ is a solution of \eqref{puntoN7}, \eqref{wave3} with
\begin{equation}\label{cond_inf}
\lim\limits_{x_2\to+\infty}\widehat{P}^\pm(\cdot,\pm x_2)= 0 \, .
\end{equation}
Then $f$ solves the second order pseudo-differential equation
\begin{equation}\label{equ_f}
\ds \left( (\tau +  iw_1\eta)^2  +  V_1^2 \eta^2
\left( \frac{8(\tau+iw_1\eta)^2}{c^2(\mu^++\mu^-)^2}  -1 \right)\right) \widehat{f} + \frac{\mu^+\mu^-}{\mu^++\mu^-}\,M =0   \, ,
\end{equation}
where $\mu^\pm=\sqrt{\left(\frac{\tau+iv_1^\pm\eta}{c}\right)^2+\eta^2}$ is
such that
$
\Re\mu^\pm>0$ if $\Re\tau>0$,
and
\begin{equation}\label{def_M}
M=M(\tau,\eta):= \frac{1}{\mu^+}\int_{0}^{+\infty}e^{-\mu^+ y}\widehat{\F}^+ (\cdot, y)\, dy -
\frac{1}{\mu^-}\int_{0}^{+\infty}e^{-\mu^- y}\widehat{\F}^- (\cdot,- y)\, dy \, .
\end{equation}
\end{theorem}
From the definition we see that the roots $ \mu^\pm $ are homogeneous functions of degree 1 in $(\tau, \eta)$. Therefore, the ratio $ (\tau+iw_1\eta)^2/(\mu^++\mu^-)^2 $ is homogeneous of degree 0. It follows that the symbol of \eqref{equ_f} is a homogeneous function of degree 2, see Remark \ref{remark52}. In this sense \eqref{equ_f} represents a second order pseudo-differential equation for $f$.

The main result of the paper is given by the following result.

\begin{theorem}\label{teoexist}
	Assume $\frac{v}{c}>\sqrt{2}$, and let $ \F^+\in L^2(\R^+;H^s_\gamma(\R^2)) , \F^-\in L^2(\R^-;H^s_\gamma(\R^2))$. There exists a unique solution $f\in H^{s+1}_\gamma(\R^2)$ of equation \eqref{equ_f} (with $w_1=0$), satisfying the estimate
	\begin{equation}\label{stimafF1}
	\gamma^3 \|f\|^2_{H^{s+1}_\gamma(\R^2)} \le C\left( \|\F^+\|^2_{L^2(\R^+;H^s_\gamma(\R^2))}+\|\F^-\|^2_{L^2(\R^-;H^s_\gamma(\R^2))}\right) , \qquad\forall \gamma\ge1\, ,
	\end{equation}
	for a suitable constant $C>0$ independent of $\F^\pm$ and $\gamma$.
\end{theorem}

See Remark \ref{ell_hyp} for a discussion about the different cases $ \frac{v}{c}\gtrless\sqrt{2} $ in relation with the classical stability analysis \cite{CS04MR2095445,FM63MR0154509,M58MR0097930,S00MR1775057}.

\subsection{Weighted Sobolev spaces and norms}\label{sec2.w}
We are going to introduce certain weighted Sobolev spaces in order to prove Theorem \ref{teoexist}.
Functions are defined over the two half-spaces $\{(t,x_1,x_2)\in\mathbb{R}^3:x_2\gtrless0\}$;
the boundary of the half-spaces is identified to $\mathbb{R}^2$.
For all $s\in\mathbb{R}$ and for all $\gamma\geq 1$,
the usual Sobolev space $H^s(\mathbb{R}^2)$ is equipped with the following norm:
\begin{align*}
\|v\|_{s,\gamma}^2:=\frac{1}{(2\pi)^2} \iint_{\mathbb{R}^2}\Lambda^{2s}(\tau,\eta) |\widehat{v}(\delta,\eta)|^2\d \delta \d\eta,\qquad
\Lambda^{s}(\tau,\eta):=(\gamma^2+\delta^2+\eta^2)^{\frac{s}{2}}=(|\tau|^2+\eta^2)^{\frac{s}{2}},
\end{align*}
where $\widehat{v}(\delta,\eta)$ is the Fourier transform of $v(t,x_1)$ and $ \tau=\gamma+i\delta $.
We will abbreviate the usual norm of $L^2(\mathbb{R}^2)$ as
\begin{align*}
\|\cdot\|:=\|\cdot\|_{0,\gamma}\, .
\end{align*}
The scalar product in $L^2(\mathbb{R}^2)$ is denoted as follows:
\begin{align*}
\langle a,b\rangle:=\iint_{\mathbb{R}^2} a(x)\overline{b(x)}\d x,
\end{align*}
where $\overline{b(x)}$ is the complex conjugation of $b(x)$.

For $s\in\mathbb{R}$ and $\gamma\geq 1$, we introduce the weighted Sobolev
space $H^{s}_{\gamma}(\mathbb{R}^2)$ as
\begin{align*}
H^{s}_{\gamma}(\mathbb{R}^2)&:=\left\{
u\in\mathcal{D}'(\mathbb{R}^2)\,:\, \mathrm{e}^{-\gamma t}u(t,x_1)\in
H^{s}(\mathbb{R}^2) \right\},
\end{align*}
and its norm $\|u\|_{H^{s}_{\gamma}(\mathbb{R}^2)}:=\|\mathrm{e}^{-\gamma t}u\|_{s,\gamma}$.
We write $L^2_{\gamma}(\mathbb{R}^2):=H^0_{\gamma}(\mathbb{R}^2)$ and $\|u\|_{L^2_{\gamma}(\mathbb{R}^2)}:=\|\mathrm{e}^{-\gamma t}u\|$.

We define $L^2(\mathbb{R}^\pm;H^{s}_{\gamma}(\mathbb{R}^2))$
as the spaces of distributions with finite norm
\begin{align*}
\|u\|_{L^2(\mathbb{R}^\pm;H^s_{\gamma}(\mathbb{R}^2))}^2:=\int_{\mathbb{R}^+}\|u(\cdot,\pm x_2)\|_{H^s_{\gamma}(\mathbb{R}^2)}^2\d x_2
\, .
\end{align*}

\section{Proof of Theorem \ref{teo_equ}}

In order to obtain an evolution equation for $f$, we will find an explicit formula for the solution $P^\pm$ of \eqref{wave3}, and substitute into \eqref{puntoN7}.
We first perform the Fourier transform of problem \eqref{puntoN7}, \eqref{wave3} and obtain
\begin{equation}\label{puntoN8}
\ds  (\tau +  iw_1\eta)^2 \widehat{f} -  V_1^2 \eta^2\widehat{f}
+\frac{c^2}2 \ddue (\widehat{P}^+ +  \widehat{P}^-)   =0 \,\qquad {\rm if} \; x_2=0 \, ,
\end{equation}
\begin{equation}\label{wave4}
\begin{cases}
(\tau +iv_1^\pm\eta)^2 \widehat{P}^\pm + c^2\eta^2 \widehat{P}^\pm -c^2\p^2_{22} \widehat{P}^\pm =\widehat{\F}^\pm  \quad &
{\rm if }\; x_2\gtrless0 \, ,\\
[\widehat{P}]=0 \, ,\\
[c^2\ddue \widehat{P} ] =-4i\eta  V_1(\tau+iw_1\eta) \widehat{f} & {\rm if} \; x_2=0 \, .
\end{cases}
\end{equation}
To solve \eqref{wave4} we take the Laplace transform in $x_2$ with dual variable $s\in\C $, defined by
\begin{equation*}
\mathcal{L}[\widehat{P}^\pm](s)=\int_0^\infty e^{-sx_2}\widehat{P}^\pm(\cdot,\pm x_2)\, dx_2\,,
\end{equation*}
\begin{equation*}
\mathcal{L}[\widehat{\F }^\pm](s)=\int_0^\infty e^{-sx_2}\widehat{\F }^\pm(\cdot,\pm x_2)\, dx_2\,.
\end{equation*}
For the sake of simplicity of notation, here we neglect the dependence on $\tau,\eta$. From \eqref{wave4} we obtain
\begin{equation*}
\left((\tau+iv^\pm_1\eta)^2+c^2\eta^2-c^2s^2\right)\mathcal{L}[\widehat{P}^\pm](s)=   \mathcal{L}[\widehat{\F }^\pm](s)  - c^2s\widehat{P}^\pm(0) \mp c^2 \ddue\widehat{P}^\pm(0)\, .
\end{equation*}
It follows that
\begin{equation}\label{laplace}
\mathcal{L}[\widehat{P}^\pm](s)=
\frac{c^2s\widehat{P}^\pm(0) \pm c^2 \ddue\widehat{P}^\pm(0)}{c^2s^2-(\tau+iv^\pm_1\eta)^2-c^2\eta^2}
- \frac{\mathcal{L}[\widehat{\F }^\pm](s)}{c^2s^2-(\tau+iv^\pm_1\eta)^2-c^2\eta^2}\, .
\end{equation}
Let us denote by $\mu^\pm=\sqrt{\left(\frac{\tau+iv_1^\pm\eta}{c}\right)^2+\eta^2}$ the root of the equation (in $s$)
\[c^2s^2-(\tau+iv^\pm_1\eta)^2-c^2\eta^2=0\,,\]
such that
\begin{equation}\label{mu}
\Re\mu^\pm>0\quad{\rm if }\quad \gamma>0 \,
\end{equation}
($\Re$ denotes the real part). We show this property in Lemma \ref{lemma_mu}.
Recalling that $\mathcal{L}[e^{\alpha x}H(x)](s)=\frac1{s-\alpha}$ for any $\alpha\in\C$, where $H(x)$ denotes the Heaviside function, we take the inverse Laplace transform of \eqref{laplace} and obtain
\begin{multline}\label{{formulaQ+}}
\widehat{P}^+(\cdot,x_2)=\widehat{P}^+(0)\cosh(\mu^+ x_2) +\ddue \widehat{P}^+(0) \frac{\sinh(\mu^+ x_2)}{\mu^+}\\ - \int_{0}^{x_2} \frac{\sinh(\mu^+ (x_2-y))}{c^2\mu^+} \widehat{\F }^+(\cdot,y)\, dy \, ,\qquad  x_2>0 \,,
\end{multline}
\begin{multline}\label{{formulaQ-}}
\widehat{P}^-(\cdot,-x_2)=\widehat{P}^-(0)\cosh(\mu^- x_2) -\ddue \widehat{P}^-(0) \frac{\sinh(\mu^- x_2)}{\mu^-}\\ - \int_{0}^{x_2} \frac{\sinh(\mu^- (x_2-y))}{c^2\mu^-} \widehat{\F }^-(\cdot,-y)\, dy \, ,\qquad  x_2>0 \, .
\end{multline}
We need to determine the values of $\widehat{P}^\pm(0) , \ddue \widehat{P}^\pm(0) $ in \eqref{{formulaQ+}}, \eqref{{formulaQ-}}. Two conditions are given by the boundary conditions in \eqref{wave4}, and two more conditions are obtained by imposing the behavior at infinity
\eqref{cond_inf}.
Recalling \eqref{mu}, under the assumption \eqref{cond_infF} it is easy to show that
\begin{equation}\label{cond_infF_int}
\lim\limits_{x_2\to+\infty}\int_{0}^{x_2}e^{-\mu^\pm(x_2-y)}\widehat{\F}^\pm (\cdot,\pm y)\, dy= 0 \, .
\end{equation}
From \eqref{cond_inf}, \eqref{{formulaQ+}}, \eqref{{formulaQ-}}, \eqref{cond_infF_int} it follows that
\begin{equation}\label{cond_inf2}
\widehat{P}^\pm(0) \pm \frac{1}{\mu^\pm} \ddue \widehat{P}^\pm(0) - \frac{1}{c^2\mu^\pm}\int_{0}^{+\infty}e^{-\mu^\pm y}\widehat{\F}^\pm (\cdot,\pm y)\, dy= 0 \, .
\end{equation}
Collecting the boundary conditions in \eqref{wave4} and \eqref{cond_inf2} gives the linear system
\begin{equation}\label{system}
\begin{cases}
\widehat{P}^+(0)-\widehat{P}^-(0)=0 \, ,\\
\ddue \widehat{P}^+(0) - \ddue \widehat{P}^-(0)  =-4i\eta \frac{ V_1}{c}\left(\frac{\tau+iw_1\eta}{c}\right) \widehat{f}
\\
\mu^+\widehat{P}^+(0) +  \ddue \widehat{P}^+(0)= \frac{1}{c^2}\int_{0}^{+\infty}e^{-\mu^+ y}\widehat{\F}^+ (\cdot, y)\, dy
\\
\mu^-\widehat{P}^-(0) -  \ddue \widehat{P}^-(0)=\frac{1}{c^2}\int_{0}^{+\infty}e^{-\mu^- y}\widehat{\F}^- (\cdot,- y)\, dy \, .
\end{cases}
\end{equation}
The determinant of the above linear system equals $\mu^++\mu^- $; from \eqref{mu} it never vanishes as long as $ \gamma>0. $ Solving \eqref{system} gives
\begin{equation}\label{somma_deriv}
\ddue \widehat{P}^+(0) + \ddue \widehat{P}^-(0)  =-4i\eta \frac{ V_1}{c}\left(\frac{\tau+iw_1\eta}{c}\right) \widehat{f}\;  \frac{\mu^+-\mu^-}{\mu^++\mu^-} + 2\frac{\mu^+\mu^-}{\mu^++\mu^-}\frac{M}{c^2} \, ,
\end{equation}
where we have set
\begin{equation*}
M:= \frac{1}{\mu^+}\int_{0}^{+\infty}e^{-\mu^+ y}\widehat{\F}^+ (\cdot, y)\, dy -
\frac{1}{\mu^-}\int_{0}^{+\infty}e^{-\mu^- y}\widehat{\F}^- (\cdot,- y)\, dy \, .
\end{equation*}
We substitute \eqref{somma_deriv} into \eqref{puntoN8} and obtain the equation for $\widehat{f}$
\begin{equation}\label{equ_f0}
\ds \left( (\tau +  iw_1\eta)^2  -  V_1^2 \eta^2
-2i { V_1} \eta\left(\tau+iw_1\eta\right) \;  \frac{\mu^+-\mu^-}{\mu^++\mu^-} \right) \widehat{f} + \frac{\mu^+\mu^-}{\mu^++\mu^-}M =0   \, .
\end{equation}
Finally, we compute
\[
\frac{\mu^+-\mu^-}{\mu^++\mu^-} =4\frac{V_1}{c^2}\frac{i\eta(\tau+iw_1\eta)}{(\mu^++\mu^-)^2},
\]
and substituting this last expression in \eqref{equ_f0} we can rewrite it as
\begin{equation*}
\ds \left( (\tau +  iw_1\eta)^2  +  V_1^2 \eta^2
\left(8 \left(\frac{\tau+iw_1\eta}{c(\mu^++\mu^-)}\right)^2  -1 \right)\right) \widehat{f} + \frac{\mu^+\mu^-}{\mu^++\mu^-}M =0   \, ,
\end{equation*}
that is \eqref{equ_f}.

\section{The symbol of the pseudo-differential equation \eqref{equ_f} for the front}

 Let us denote the symbol of \eqref{equ_f} by $ \Sigma: $
 \[ \Sigma=\Sigma(\tau,\eta):= (\tau +  iw_1\eta)^2  +  V_1^2 \eta^2
 \left( \frac{8(\tau+iw_1\eta)^2}{c^2(\mu^+(\tau,\eta)+\mu^-(\tau,\eta))^2}  -1 \right).\]
In order to take the homogeneity into account, we define the hemisphere:
\begin{align*}
\Xi_1:=\left\{(\tau,\eta)\in \mathbb{C}\times\mathbb{R}\, :\,
|\tau|^2+\eta^2=1,\Re \tau\geq 0 \right\},
\end{align*}
and the set of ``frequencies'':
\begin{align*}
\Xi:=\left\{(\tau,\eta)\in \mathbb{C}\times\mathbb{R}\, :\,
\Re \tau\geq 0, (\tau,\eta)\ne (0,0) \right\}=(0,\infty)\cdot\Xi_1  \,.
\end{align*}
From now on we assume
\[ v^+_1=v>0, \qquad v^-_1=-v \,, \]
so that \[ w_1=0, \qquad V_1=v\, . \]
From this assumption it follows that
\begin{equation}\label{def_Sigma}
\Sigma(\tau,\eta)= \tau^2  +  v^2 \eta^2
\left(8\left( \frac{\tau/c}{\mu^+(\tau,\eta)+\mu^-(\tau,\eta)}\right)^2  -1 \right).
\end{equation}

\subsection{Study of the roots $\mu^\pm$}
\begin{lemma}\label{lemma_mu}
Let $ (\tau,\eta)\in\Xi $ and let us consider the equation
\begin{equation}\label{equ_mu}
s^2=
\left(\frac{\tau \pm iv\eta}{c}\right)^2+\eta^2.
\end{equation}
For both cases $\pm$ of \eqref{equ_mu} there exists one root, denoted by $ \mu^\pm=\mu^\pm(\tau,\eta) $, such that $ \Re\mu^\pm>0 $ as long as $ \Re\tau>0 $. The other root is $ -\mu^\pm $. The roots $ \mu^\pm$ admit a continuous extension to points  $ (\tau,\eta)=(i\delta,\eta)\in\Xi $, i.e. with $ \Re\tau=0 $.
Specifically we have:

(i) if $ \eta=0 $, $ \mu^\pm(i\delta,0)=i\delta/c $ ;

(ii) if $ \eta\not=0 $,
\begin{equation}\label{mu+}
\begin{array}{ll}
\mu^+(i\delta,\eta)=\sqrt{-\left(\frac{\delta +v\eta}{c}\right)^2+\eta^2} \qquad &{\it if } \; -\left(\frac{v}{c}+1\right) <\frac{\delta}{c\eta}<-\left(\frac{v}{c}-1\right)\, , \\
\mu^+(i\delta,\eta)=0 \qquad &{\it if } \; \frac{\delta}{c\eta}=-\left(\frac{v}{c}\pm 1\right)\, ,
\\
\mu^+(i\delta,\eta)=-i\sgn(\eta) \sqrt{\left(\frac{\delta +v\eta}{c}\right)^2-\eta^2} \qquad &{\it if }\;  \frac{\delta}{c\eta}<-\left(\frac{v}{c}+1\right)\, , \\
\mu^+(i\delta,\eta)=i\sgn(\eta) \sqrt{\left(\frac{\delta +v\eta}{c}\right)^2-\eta^2} \qquad &{\it if }\;  \frac{\delta}{c\eta}>-\left(\frac{v}{c}-1\right)\, ,
\end{array}
\end{equation}
and
\begin{equation}\label{mu-}
\begin{array}{ll}
\mu^-(i\delta,\eta)=\sqrt{-\left(\frac{\delta -v\eta}{c}\right)^2+\eta^2} \qquad &{\it if } \; \frac{v}{c}-1 <\frac{\delta}{c\eta}<\frac{v}{c}+1\, , \\
\mu^-(i\delta,\eta)=0 \qquad &{\it if } \; \frac{\delta}{c\eta}=\frac{v}{c}\pm 1\, ,
\\
\mu^-(i\delta,\eta)=-i\sgn(\eta) \sqrt{\left(\frac{\delta -v\eta}{c}\right)^2-\eta^2} \qquad &{\it if }\;  \frac{\delta}{c\eta}<\frac{v}{c}-1\, , \\
\mu^-(i\delta,\eta)=i\sgn(\eta) \sqrt{\left(\frac{\delta -v\eta}{c}\right)^2-\eta^2} \qquad &{\it if }\;  \frac{\delta}{c\eta}>\frac{v}{c}+1\, .
\end{array}
\end{equation}
\end{lemma}
\begin{proof}
(i) If $ \eta=0 $, \eqref{equ_mu} reduces to $s^2=(\tau/c)^2$. We choose $\mu^\pm=\tau/c$ which has $\Re\mu^\pm>0$ if $\Re\tau=\gamma>0$; obviously, the continuous extension for $\Re\tau=0$ is $\mu^\pm=i\delta/c$.
(ii) Assume $ \eta\not=0.$ Let us denote
\begin{equation*}
\alpha^\pm:=\left(\frac{\tau \pm iv\eta}{c}\right)^2+\eta^2=\frac{\gamma^2-(\delta \pm v\eta)^2+c^2\eta^2}{c^2} +2i\gamma\frac{\delta\pm v\eta}{c^2}\, .
\end{equation*}
For $\gamma>0$,
$\Im\alpha^\pm=0$ if and only if $\delta\pm v\eta=0$,
and
$\alpha^\pm_{|\delta\pm v\eta=0}=(\gamma/c)^2+\eta^2>0$. It follows that either $\alpha^\pm\in\R, \alpha^\pm>0$, or $\alpha^\pm\in\C$ with $\Im\alpha^\pm\not=0$.
In both cases $\alpha^\pm$ has two square roots, one with strictly positive real part (that we denote by $\mu^\pm$), the other one with strictly negative real part. For the continuous extension in points with $\Re\tau=0,$ we have
\begin{equation}\label{caso+}
\mu^\pm(i\delta,\eta)=\sqrt{-\left(\frac{\delta \pm v\eta}{c}\right)^2+\eta^2} \qquad {\rm if } \quad -\left(\frac{\delta \pm v\eta}{c}\right)^2+\eta^2\ge0\, ,
\end{equation} and
\begin{equation}\label{caso-}
\mu^\pm(i\delta,\eta)=i\sgn(\delta\pm v\eta)\sqrt{\left(\frac{\delta \pm v\eta}{c}\right)^2-\eta^2} \qquad {\rm if } \quad -\left(\frac{\delta \pm v\eta}{c}\right)^2+\eta^2<0\, .
\end{equation}
We also observe that
\begin{equation}\label{caso+2}
\begin{array}{ll}
\sgn(\delta+v\eta)=-\sgn(\eta) \qquad{\rm if } \quad \frac{\delta}{c\eta}<-\left(\frac{v}{c}+1\right) ,\\
\sgn(\delta+v\eta)=\sgn(\eta)\qquad{\rm if } \quad \frac{\delta}{c\eta}>-(\frac{v}{c}-1) ,
\end{array}
\end{equation}
and
\begin{equation}\label{caso-2}
\begin{array}{ll}
\sgn(\delta-v\eta)=-\sgn(\eta) \qquad{\rm if } \quad \frac{\delta}{c\eta}<\frac{v}{c}-1 ,\\
\sgn(\delta-v\eta)=\sgn(\eta)\qquad{\rm if } \quad \frac{\delta}{c\eta}>\frac{v}{c}+1 .
\end{array}
\end{equation}
From \eqref{caso+}--\eqref{caso-2} we obtain \eqref{mu+}, \eqref{mu-}.
\end{proof}


\begin{corollary}\label{coroll_mu}
From \eqref{mu+}, \eqref{mu-} the roots $ \mu^\pm $ only vanish in points $ (\tau,\eta)=(i\delta,\eta) $ with
\[ \delta=-(v\pm c) \eta \qquad\forall\eta\not=0 \qquad ({\rm where}\;\mu^+=0)\, ,\]
or
\[ \delta=(v\pm c) \eta   \qquad\forall\eta\not=0\qquad ({\rm where}\;\mu^-=0)\, .\]
If $v\not= c $, the above four families of points
\[
\delta=-(v+ c) \eta, \quad \delta=-(v- c)\eta, \quad \delta=(v- c) \eta, \quad \delta=(v+ c) \eta \, ,
 \] are always mutually distinct. If $ v=c $, the two families in the middle coincide and we have
 \begin{equation}\label{sommavc}
\mu^+(-2ic\eta,\eta)=0, \qquad \mu^+(0,\eta)=\eta^-(0,\eta)=0, \qquad  \mu^-(2ic\eta,\eta)=0, \qquad\forall\eta\not=0\, .
 \end{equation}
\end{corollary}

From \eqref{def_Sigma}, the symbol $ \Sigma $ is not defined in points $ (\tau,\eta)\in\Xi$ where $\mu^++\mu^-$
vanishes. From Lemma \ref{lemma_mu} we already know that $\Re\mu^\pm>0$ in all points with $\Re\tau>0$. It follows that $\Re(\mu^++\mu^-)>0$ and thus $\mu^++\mu^-\not=0$ in all such points. Therefore the symbol is defined for $\Re\tau>0$.

It rests to study if $\mu^++\mu^-$
vanishes in points $ (\tau,\eta)=(i\delta,\eta) $ with $ \Re\tau=0 $. From Corollary \ref{coroll_mu} we obtain that if $v\not= c $ then $\mu^++\mu^-\not=0$ in all points with $\delta=-(v\pm c) \eta$ and $\delta=(v\pm c) \eta$ (in these points, if $\mu^+=0$ then $\mu^-\not=0$ and viceversa). If $v= c $, then $\mu^+(0,\eta)+\mu^-(0,\eta)=0$.

From now on we adopt the usual terminology: $ v>c $ is the {\it supersonic} case, $ v<c $ is the {\it subsonic} case, $ v=c $ is the {\it sonic} case. The next lemma regards the supersonic case.

\begin{lemma}[$ v>c $]\label{super_mu}
Let $ (\tau,\eta)=(i\delta,\eta)\in\Xi $ such that $ \Re\tau=0, \eta\not=0 $. For all such points the following facts hold.
\begin{itemize}
\item[(i)] If $\frac{\delta}{c\eta}<-\left(\frac{v}{c}+1\right)$ then $\mu^+\in i\R$, $\mu^-\in i\R$, $\mu^++\mu^-\not=0$, $\Re(\mu^+\mu^-)<0$;
\\
\item[(ii)] If $-\left(\frac{v}{c}+1\right) <\frac{\delta}{c\eta}<-\left(\frac{v}{c}-1\right)
$ then $\mu^+\in \R^+$, $\mu^-\in i\R$, $\mu^++\mu^-\not=0$, $\Re(\mu^+\mu^-)=0$;
\\
\item[(iii)] If $-\left(\frac{v}{c}-1\right) <\frac{\delta}{c\eta}< \frac{v}{c}-1
$ then $\mu^+\in i\R$, $\mu^-\in i\R$, $\mu^+(0,\eta)+\mu^-(0,\eta)=0$, $\Re(\mu^+\mu^-)>0$;
\\
\item[(iv)] If $\frac{v}{c}-1 <\frac{\delta}{c\eta}< \frac{v}{c}+1
$ then $\mu^+\in i\R$, $\mu^-\in \R^+$, $\mu^++\mu^-\not=0$, $\Re(\mu^+\mu^-)=0$;
\\
\item[(v)] If $\frac{\delta}{c\eta}> \frac{v}{c}+1
$ then $\mu^+\in i\R$, $\mu^-\in i\R$, $\mu^++\mu^-\not=0$, $\Re(\mu^+\mu^-)<0$.

\end{itemize}
\end{lemma}
We emphasize that the above properties hold in all points $(i\delta,\eta)$ as indicated according to the value of ${\delta}/({c\eta})$, except for the case (iii) where $\mu^++\mu^-=0$ if and only if $(\delta,\eta)=(0,\eta).$
From \eqref{sommavc} and (iii) we have
\begin{equation}\label{somma_mu}
{\rm if }\;\, v\ge c \qquad \mu^+(0,\eta)+\mu^-(0,\eta)=0 \qquad \forall\eta\not=0\, .
\end{equation}
\begin{proof}[Proof of Lemma \ref{super_mu}]
(i) If $\frac{\delta}{c\eta}<-\left(\frac{v}{c}+1\right)$ then $\mu^+\in i\R$ follows directly from $\eqref{mu+}_3$ and $\mu^-\in i\R$ follows from $\eqref{mu-}_3$ because $\frac{\delta}{c\eta}<\frac{v}{c}-1$. Moreover, from \eqref{mu+}, \eqref{mu-} we have
\[
\mu^++\mu^-=-i\sgn(\eta)\left(\sqrt{\left(\frac{\delta +v\eta}{c}\right)^2-\eta^2}+\sqrt{\left(\frac{\delta -v\eta}{c}\right)^2-\eta^2} \right)\not=0 \,,
 \]
 \[
 \Re(\mu^+\mu^-)=- \sqrt{\left(\left(\frac{\delta +v\eta}{c}\right)^2-\eta^2 \right) \left( \left(\frac{\delta -v\eta}{c}\right)^2-\eta^2\right)} <0\,.
  \]
The cases (ii) and (iv) follow directly from \eqref{mu+}, \eqref{mu-}.

(iii) If $-\left(\frac{v}{c}-1\right) <\frac{\delta}{c\eta}< \frac{v}{c}-1
$ then $\mu^\pm\in i\R$ follows from \eqref{mu+}, \eqref{mu-}. Moreover it holds
\[
\mu^++\mu^-=i\sgn(\eta)\left(\sqrt{\left(\frac{\delta +v\eta}{c}\right)^2-\eta^2}-\sqrt{\left(\frac{\delta -v\eta}{c}\right)^2-\eta^2} \right)=0
\quad\mbox{if and only if }\, \delta=0 \,,
\]  recalling that here $\eta\not=0$.
It follows that $\mu^+(0,\eta)+\mu^-(0,\eta)=0$ for all $\eta\not=0$. We also have
 \[
\Re(\mu^+\mu^-)= \sqrt{\left(\left(\frac{\delta +v\eta}{c}\right)^2-\eta^2 \right) \left( \left(\frac{\delta -v\eta}{c}\right)^2-\eta^2\right)} >0\,.
\]
The proof of case (v) is similar to the proof of case (i).
\end{proof}

The next lemma regards the subsonic case.

\begin{lemma}[$ v<c $]\label{sub_mu}
	Let $ (\tau,\eta)=(i\delta,\eta)\in\Xi $ such that $ \Re\tau=0, \eta\not=0 $. For all such points the following facts hold.
	\begin{itemize}
		\item[(i)] If $\frac{\delta}{c\eta}<-\left(\frac{v}{c}+1\right)$ then $\mu^+\in i\R$, $\mu^-\in i\R$, $\mu^++\mu^-\not=0$, $\Re(\mu^+\mu^-)<0$;
		\\
		\item[(ii)] If $-\left(\frac{v}{c}+1\right) <\frac{\delta}{c\eta}<\frac{v}{c}-1
		$ then $\mu^+\in \R^+$, $\mu^-\in i\R$, $\mu^++\mu^-\not=0$, $\Re(\mu^+\mu^-)=0$;
		\\
		\item[(iii)] If $\frac{v}{c}-1 <\frac{\delta}{c\eta}< -\left(\frac{v}{c}-1\right)
		$ then $\mu^+\in \R^+$, $\mu^-\in \R^+$, $\mu^++\mu^-\not=0$, $\Re(\mu^+\mu^-)>0$;
		\\
		\item[(iv)] If $-\left(\frac{v}{c}-1\right) <\frac{\delta}{c\eta}< \frac{v}{c}+1
		$ then $\mu^+\in i\R$, $\mu^-\in \R^+$, $\mu^++\mu^-\not=0$, $\Re(\mu^+\mu^-)=0$;
		\\
		\item[(v)] If $\frac{\delta}{c\eta}> \frac{v}{c}+1
		$ then $\mu^+\in i\R$, $\mu^-\in i\R$, $\mu^++\mu^-\not=0$, $\Re(\mu^+\mu^-)<0$.
		
	\end{itemize}
\end{lemma}
\begin{proof}
The proof is similar to the proof of Lemma \ref{super_mu} and so we omit the details.
\end{proof}

\begin{lemma}[$ v=c $]\label{eq_mu}
	Let $ (\tau,\eta)=(i\delta,\eta)\in\Xi $ such that $ \Re\tau=0, \eta\not=0 $. For all such points the following facts hold.
	\begin{itemize}
		\item[(i)] If $\frac{\delta}{c\eta}<-2$ then $\mu^+\in i\R$, $\mu^-\in i\R$, $\mu^++\mu^-\not=0$, $\Re(\mu^+\mu^-)<0$;
		\\
		\item[(ii)] If $-2<\frac{\delta}{c\eta}<0$ then $\mu^+\in \R^+$, $\mu^-\in i\R$, $\mu^++\mu^-\not=0$, $\Re(\mu^+\mu^-)=0$;
		\\
		\item[(iii)] If $0 <\frac{\delta}{c\eta}< 2
		$ then $\mu^+\in i\R$, $\mu^-\in \R^+$, $\mu^++\mu^-\not=0$, $\Re(\mu^+\mu^-)=0$;
		\\
		\item[(iv)] If $\frac{\delta}{c\eta}> 2
		$ then $\mu^+\in i\R$, $\mu^-\in i\R$, $\mu^++\mu^-\not=0$, $\Re(\mu^+\mu^-)<0$.
		
	\end{itemize}
\end{lemma}
\begin{proof}
The proof is similar to the proof of Lemma \ref{super_mu} and so we omit the details.
\end{proof}


\begin{corollary}\label{zerosommamu}
From Lemma \ref{super_mu}, see also \eqref{somma_mu}, and Lemma \ref{sub_mu} it follows that $\mu^++\mu^-=0$ at $ (\tau,\eta)\in\Xi$ if and only if $\tau=0$ and $v\ge c$.
\end{corollary}
For $v\ge c$, though $\mu^++\mu^-=0$ at $ (0,\eta)\in\Xi$ , nevertheless we can define $\Sigma (0,\eta)$ by continuous extension, see Lemma \ref{extend}.

We are also interested to know if the difference $\mu^+-\mu^-$ vanishes somewhere.
\begin{lemma}\label{diff_mu}
Let $ (\tau,\eta)\in\Xi$. Then $\mu^+(\tau,\eta)=\mu^-(\tau,\eta)$ if and only if:
	\begin{itemize}
		\item[(i)] $ (\tau,\eta)= (\tau,0) $,
		\item[(ii)] $ (\tau,\eta)= (0,\eta)$, and $v\le c$.
	\end{itemize}
\end{lemma}
\begin{proof}
From \eqref{equ_mu} we obtain that $(\mu^+)^2
=(\mu^-)^2$ if and only if $\eta=0$ or $\tau=0$. If $\eta=0$ then $\mu^+
=\mu^-=\tau/c$ which gives the first case. If $\tau=0$ then $(\mu^+)^2
=(\mu^-)^2=(1-(v/c)^2)\eta^2$. For $1-(v/c)^2<0$ we obtain from Lemma \ref{lemma_mu} $\mu^\pm=\pm i\eta\sqrt{(v/c)^2-1}$ which yields $\mu^+-\mu^-=2i\eta\sqrt{(v/c)^2-1}\not=0$. For $1-(v/c)^2\ge 0$ we obtain $\mu^\pm= \sqrt{1-(v/c)^2}|\eta|$, that is the second case.
\end{proof}
From Lemma \ref{lemma_mu} we know that the roots $\mu^\pm$ satisfy $\Re\mu^\pm>0$ if $\Re\tau=\gamma>0$. Actually we can prove more than that.
\begin{lemma}\label{stima_Re_mu}
	Let $ (\tau,\eta)\in\Xi$ with $\Re\tau=\gamma>0$. Then
\begin{equation}\label{est_Re_mu}
\Re\mu^\pm(\tau,\eta)\ge \frac{1}{\sqrt{2}\,c}\,\gamma\, .
	\end{equation}

\end{lemma}
\begin{proof}
We consider $\mu^+$.
From \eqref{equ_mu} we obtain
\[
(\Re\mu^+)^2-(\Im\mu^+)^2=\frac{1}{c^2}(\gamma^2-(\delta+v\eta)^2)+\eta^2, \qquad
\Re\mu^+\Im\mu^+=\frac{1}{c^2}\gamma(\delta+v\eta) \,.
 \]
Since $\Re\mu^+>0$ for $\gamma>0$, we can divide by $\Re\mu^+$ the second equation, then substitute the value of $\Im\mu^+$ into the first one and obtain
\[
(\Re\mu^+)^4+\alpha(\Re\mu^+)^2+\beta=0\,,
 \]
 where we have set
 \[
 \alpha=-\left(\frac{1}{c^2}(\gamma^2-(\delta+v\eta)^2)+\eta^2
 \right)\,, \qquad \beta=-\frac{1}{c^4}\gamma^2(\delta+v\eta)^2 \le0 \,.
  \]
  We show that $\alpha^2-4\beta>0$ for $\gamma>0$, and obtain
  \[
 2 (\Re\mu^+)^2=-\alpha+\sqrt{\alpha^2-4\beta}\ge
 (\gamma/c)^2+\left|\frac{1}{c^2}(\delta+v\eta)^2-\eta^2
 \right| -\left(\frac{1}{c^2}(\delta+v\eta)^2-\eta^2
 \right) \ge (\gamma/c)^2 	\,,
   \]
 which gives \eqref{est_Re_mu} for  $\mu^+$. The proof for $\mu^-$ is similar.
\end{proof}

\subsection{Study of the symbol $\Sigma$}

The next lemma regards the continuous extension of the symbol in points $(0,\eta)$ where $\mu^++\mu^-$ vanishes, see Corollary \ref{zerosommamu}. We only consider the case $v\geq c$.

\begin{lemma}\label{extend}
Assume $v\geq c$. Let $(\tau,\eta)\in\Xi$ with $ \Re\tau\ge0 $ and $\bar{\eta}\not=0$ fixed. Then
\begin{equation}\label{extension}\lim\limits_{(\tau,\eta)\to(0,\bar{\eta})}\left( \frac{\tau/c}{\mu^++\mu^-} \right)^2=\frac{(v/c)^2-1}{4(v/c)^2}  \, .
\end{equation}
\end{lemma}
\begin{proof}

{\it First case: $v>c$}.
We first consider the case $ \Re\tau>0. $ Let $ \tau=\gamma+i\delta$ with $0<\gamma\ll1, \delta\in\R.$ Then
\[
\left(\frac{\tau \pm iv\eta}{c}\right)^2+\eta^2 =a_\pm+ib_\pm \,,
 \]
 with
\begin{equation}\label{def_ab}
a_\pm=\left(\frac{\gamma}{c}\right)^2-  \left(\frac{\delta\pm v\eta}{c}\right)^2+\eta^2 , \quad b_\pm=2\gamma\frac{\delta\pm v\eta}{c^2} \,.
\end{equation}

For the computation of the square roots of $a_\pm+ib_\pm$ it is useful to recall that the square roots of the complex number $a+ib$ ($a,b$ real) are
\begin{equation}\label{roots}
  \pm\left\{\sgn(b)\sqrt{\frac{r+a}{2}} +i\sqrt{\frac{r-a}{2}}\right\}, \qquad r=|a+ib|
\end{equation}
(by convention $\sgn(0)=1$). In our case we compute
\begin{equation}\label{rpm}
r_\pm^2:=a_\pm^2+b_\pm^2
=
\left[ \left(\frac{\gamma}{c}\right)^2+  \left(\left|  \frac{\delta\pm v\eta}{c} \right| - |\eta|\right)^2\right]
\left[ \left(\frac{\gamma}{c}\right)^2+  \left(\left|  \frac{\delta\pm v\eta}{c} \right| + |\eta|\right)^2\right] \, .
\end{equation}
Substituting the definition of $ a_\pm, b_\pm $ in \eqref{def_ab}, $ r_\pm $ in \eqref{rpm}, into \eqref{roots} and taking the limit as $\gamma\downarrow 0, \delta\to\widebar\delta, \eta\to\widebar\eta$, with $ (\widebar\delta, \widebar\eta) \not=(0,0)$ we can prove again the formulas \eqref{mu+}, \eqref{mu-} of continuous extension of $ \mu^\pm $ to points with $ \Re\tau=0 $.

Let us study the limit of $ \mu^++\mu^- $ as $\gamma\downarrow 0, \delta\to0, \eta\to\widebar\eta$, with $  \widebar\eta\not=0$. By continuity, for $ (\gamma,\delta,\eta) $ sufficiently close to $ (0,0,\widebar\eta) $, we have from \eqref{def_ab} that $ \sgn(b_+)=\sgn(\delta+v\eta) =\sgn(\widebar\eta)$. If $ \widebar\eta>0$, from \eqref{roots} it follows that
\begin{equation*}
\mu^+= \sqrt{\frac{r_++a_+}{2}} +i\sqrt{\frac{r_+-a_+}{2}}\,.
\end{equation*}
With similar considerations we get
\begin{equation*}
\mu^-= \sqrt{\frac{r_-+a_-}{2}} -i\sqrt{\frac{r_--a_-}{2}}\,.
\end{equation*}
If $ \widebar\eta<0$, from \eqref{roots} it follows that
\begin{equation*}
\mu^+= \sqrt{\frac{r_++a_+}{2}} -i\sqrt{\frac{r_+-a_+}{2}}\,, \qquad
\mu^-= \sqrt{\frac{r_-+a_-}{2}} +i\sqrt{\frac{r_--a_-}{2}}\,.
\end{equation*}
Thus,
\begin{equation}\label{sommamu}
\mu^++\mu^-= \sqrt{\frac{r_++a_+}{2}}+ \sqrt{\frac{r_-+a_-}{2}} +i\sgn(\widebar\eta)\left(\sqrt{\frac{r_+-a_+}{2}} -\sqrt{\frac{r_--a_-}{2}}\right)\,.
\end{equation}
From \eqref{def_ab}, \eqref{rpm} we obtain (recall that $ v>c $)
\begin{equation}\label{lim_diff_ra}
\lim\limits_{(\gamma,\delta,\eta) \to(0,0,\widebar\eta)}(r_\pm-a_\pm)=2\left(\left(\frac{v}{c}\right)^2-1\right)\widebar\eta^2	\,,
\end{equation}
\begin{equation}\label{sommara}
r_\pm+a_\pm=\frac{r_\pm^2-a_\pm^2}
{r_\pm-a_\pm}=\frac{b_\pm^2}
{r_\pm-a_\pm}=\frac{4\left(\frac{\gamma}{c}\right)^2\left(\frac{\delta\pm v\eta}{c}\right)^2}
{r_\pm-a_\pm} \,.
\end{equation}
From \eqref{sommamu}, \eqref{sommara}, the real part of $ \mu^++\mu^- $ is given by
\begin{equation}\label{Resommamu}
\Re(\mu^++\mu^-)=\sqrt{\frac{r_++a_+}{2}}+ \sqrt{\frac{r_-+a_-}{2}}=\sqrt{2}\, \frac{\gamma}{c} \, \left(\frac{\left| \frac{\delta+ v\eta}{c} \right|}{\sqrt{r_+-a_+}}  + \frac{\left| \frac{\delta- v\eta}{c} \right|}{\sqrt{r_--a_-}} \right) \, .
\end{equation}
From \eqref{lim_diff_ra}, \eqref{Resommamu} it follows that
\begin{equation}\label{lim_Re_somma_mu}
\Re(\mu^++\mu^-)= \frac{\gamma}{c} \, \left(\frac{2\frac{v}{c}}{\sqrt{\left(\frac{v}{c}\right)^2-1}}+o(1)\right) \qquad{\rm as}\quad (\gamma,\delta,\eta) \to(0,0,\widebar\eta)\, .
\end{equation}
Now we consider the imaginary part of $ \mu^++\mu^- $
\begin{equation}\label{Imsommamu}
\Im(\mu^++\mu^-)=\sgn(\widebar\eta)\left(\sqrt{\frac{r_+-a_+}{2}} -\sqrt{\frac{r_--a_-}{2}}\right)=
\frac{\sgn(\widebar\eta)}{\sqrt{2}}\frac{(r_+-a_+)-(r_--a_-)}{\sqrt{r_+-a_+} +\sqrt{r_--a_-}} \,.
\end{equation}
Using \eqref{def_ab}, \eqref{rpm} gives
\begin{equation}\label{diff_ra}
(r_+-a_+)-(r_--a_-)=\frac{r_+^2-r_-^2}{r_++r_-}+4\frac{v}{c^2}\delta\eta \,,
\end{equation}
and
\begin{equation}\label{diff_ra2}
r_+^2-r_-^2=8\frac{v}{c^2}\delta\eta\left[\left(\frac{\gamma}{c}\right)^2+\left(\frac{\delta}{c}\right)^2+ \left(\left(\frac{ v}{c}\right)^2
-1\right)\eta^2
\right] \,.
\end{equation}
Moreover it holds
\begin{equation}\label{lim_somma_r}
\lim\limits_{(\gamma,\delta,\eta) \to(0,0,\widebar\eta)}(r_++r_-)=2\left(\left(\frac{v}{c}\right)^2-1\right)\widebar\eta^2	\,.
\end{equation}
Combining \eqref{lim_diff_ra}, \eqref{Imsommamu}--\eqref{lim_somma_r} gives
\begin{equation}\label{lim_Im_somma_mu}
\Im(\mu^++\mu^-)= \frac{\delta}{c} \, \left(\frac{2\frac{v}{c}}{\sqrt{\left(\frac{v}{c}\right)^2-1}}+o(1)\right) \qquad{\rm as}\quad (\gamma,\delta,\eta) \to(0,0,\widebar\eta)\, .
\end{equation}
From \eqref{lim_Re_somma_mu}, \eqref{lim_Im_somma_mu} we deduce
\begin{equation}\label{lim_somma_mu}
\mu^++\mu^-= \frac{\tau}{c} \, \left(\frac{2\frac{v}{c}}{\sqrt{\left(\frac{v}{c}\right)^2-1}}+o(1)\right) \qquad{\rm as}\quad (\gamma,\delta,\eta) \to(0,0,\widebar\eta)\, .
\end{equation}
Thus, it follows that
\begin{equation*}
\lim\limits_{(\gamma,\delta,\eta) \to(0,0,\widebar\eta)}\left( \frac{\tau/c}{\mu^++\mu^-} \right)^2=\frac{(v/c)^2-1}{4(v/c)^2}  \, ,
\end{equation*}
that is \eqref{extension}.

Consider now the case $ \Re\tau=0 $, that is $ \tau=i\delta $, and let us assume with no loss of generality that $ \widebar\eta>0.$ For $ (\delta,\eta) $ in a small neighborhood of $ (0,\widebar\eta) $ we have $ -(v-c)\eta<\delta<(v-c)\eta $, that is $ -(\frac{v}{c}-1)<\frac{\delta}{c\eta}<(\frac{v}{c}-1) $. From Lemma \ref{lemma_mu} (see also the proof of Lemma \ref{super_mu} (iii)) we get
\begin{equation*}
\frac{\tau/c}{\mu^++\mu^-}=\frac{\delta/c}{\sqrt{\left(\frac{\delta +v\eta}{c}\right)^2-\eta^2}-\sqrt{\left(\frac{\delta -v\eta}{c}\right)^2-\eta^2}}=
\frac{\sqrt{\left(\frac{\delta +v\eta}{c}\right)^2-\eta^2}+\sqrt{\left(\frac{\delta -v\eta}{c}\right)^2-\eta^2}}{4v\eta/c} \, .
\end{equation*}
Passing to the limit as $ (\delta,\eta) \mapsto (0,\widebar\eta) $ we obtain again \eqref{extension}.
The proof is complete.

\smallskip
{\it Second case: $v=c$}. Again we first consider the case $\Re\tau >0$: hence $\tau =\gamma + i\delta$ with $0<\gamma\ll1$, $\delta\in \mathbb{R}$. For $(\gamma,\delta,\eta)$ sufficiently close to $(0,0,\widebar\eta)$, with $\widebar\eta\neq0$, $\mu^+ + \mu^-$ is given by  \eqref{sommamu}, where $a_{\pm}$, $b_{\pm}$ and $r_{\pm}$ are computed in \eqref{def_ab} and in \eqref{rpm} with $v=c$.  Hence
\begin{equation}\label{mod_somma}
\begin{split}
\vert \mu^+ &+\mu^-\vert^2=r_++r_{-}+\sqrt{r_++a_+}\sqrt{r_-+a_-}-\sqrt{r_+-a_+}\sqrt{r_--a_-}\\
&=\left\vert\frac{\tau}{c}\right\vert(\alpha_++\alpha_-)+\sqrt{\left\vert\frac{\tau}{c}\right\vert\left(\alpha_++\left\vert\frac{\tau}{c}\right\vert\right)-\beta_+}\,\sqrt{\left\vert\frac{\tau}{c}\right\vert\left(\alpha_-+\left\vert\frac{\tau}{c}\right\vert\right)-\beta_-}\\
&\qquad -\sqrt{\left\vert\frac{\tau}{c}\right\vert\left(\alpha_+-\left\vert\frac{\tau}{c}\right\vert\right)+\beta_+}\,\sqrt{\left\vert\frac{\tau}{c}\right\vert\left(\alpha_--\left\vert\frac{\tau}{c}\right\vert\right)+\beta_-}\,,
\end{split}
\end{equation}
where we have set:
\begin{equation}\label{alphabeta}
\alpha_\pm=\alpha_\pm(\tau,\eta):=\sqrt{\left\vert\frac{\tau}{c}\right\vert^2+4\eta\left(\eta\pm\frac{\delta}{c}\right)}\,,\quad\beta_\pm=\beta_\pm(\delta,\eta):=2\frac{\delta}{c}\left(\frac{\delta}{c}\pm\eta\right)\,.
\end{equation}
Assume that $\overline\eta>0$, so that $\eta>0$ when it is sufficiently close to $\overline\eta$. For $\delta>0$ sufficiently close to zero, we have $\beta_-<0$ thus
\begin{equation*}
\sqrt{\left\vert\frac{\tau}{c}\right\vert\left(\alpha_++\left\vert\frac{\tau}{c}\right\vert\right)-\beta_+}\,\sqrt{\left\vert\frac{\tau}{c}\right\vert\left(\alpha_-+\left\vert\frac{\tau}{c}\right\vert\right)-\beta_-}\ge \sqrt{\left\vert\frac{\tau}{c}\right\vert\left(\alpha_++\left\vert\frac{\tau}{c}\right\vert\right)-\beta_+}\,\sqrt{\left\vert\frac{\tau}{c}\right\vert\left(\alpha_-+\left\vert\frac{\tau}{c}\right\vert\right)}
\end{equation*}
and
\begin{equation*}
\sqrt{\left\vert\frac{\tau}{c}\right\vert\left(\alpha_+-\left\vert\frac{\tau}{c}\right\vert\right)+\beta_+}\,\sqrt{\left\vert\frac{\tau}{c}\right\vert\left(\alpha_--\left\vert\frac{\tau}{c}\right\vert\right)+\beta_-}\le \sqrt{\left\vert\frac{\tau}{c}\right\vert\left(\alpha_+-\left\vert\frac{\tau}{c}\right\vert\right)+\beta_+}\,\sqrt{\left\vert\frac{\tau}{c}\right\vert\left(\alpha_--\left\vert\frac{\tau}{c}\right\vert\right)}\,;
\end{equation*}
moreover from $\beta_+=2\frac{\delta}{c}\left(\frac{\delta}{c}+\eta\right)\le 2\left\vert\frac{\tau}{c}\right\vert\left(\frac{\delta}{c}+\eta\right)$ we have
\begin{equation*}
\sqrt{\left\vert\frac{\tau}{c}\right\vert\left(\alpha_++\left\vert\frac{\tau}{c}\right\vert\right)-\beta_+}\ge\sqrt{\left\vert\frac{\tau}{c}\right\vert\left(\alpha_++\left\vert\frac{\tau}{c}\right\vert\right)-2\left\vert\frac{\tau}{c}\right\vert\left(\frac{\delta}{c}+\eta\right)}
\end{equation*}
and
\begin{equation*}
\sqrt{\left\vert\frac{\tau}{c}\right\vert\left(\alpha_+-\left\vert\frac{\tau}{c}\right\vert\right)+\beta_+}\le\sqrt{\left\vert\frac{\tau}{c}\right\vert\left(\alpha_+-\left\vert\frac{\tau}{c}\right\vert\right)+2\left\vert\frac{\tau}{c}\right\vert\left(\frac{\delta}{c}+\eta\right)}\,.
\end{equation*}
We use the last inequalities in \eqref{mod_somma} to find
\begin{equation*}
\begin{split}
\vert\mu^++\mu^-\vert^2\ge \left\vert\frac{\tau}{c}\right\vert\Theta(\tau,\eta)\,,
\end{split}
\end{equation*} 
where
\begin{equation*}
\Theta(\tau,\eta):=\alpha_++\alpha_-+\sqrt{\alpha_++\left\vert\frac{\tau}{c}\right\vert-2\left(\frac{\delta}{c}+\eta\right)}\,\sqrt{\alpha_++\left\vert\frac{\tau}{c}\right\vert}-\sqrt{\alpha_+-\left\vert\frac{\tau}{c}\right\vert+2\left(\frac{\delta}{c}+\eta\right)}\,\sqrt{\alpha_--\left\vert\frac{\tau}{c}\right\vert}
\end{equation*}
satisfies
\begin{equation}\label{lim_theta}
\lim\limits_{(\tau,\eta)\to (0,\overline\eta)}\Theta(\tau,\eta)=2(2-\sqrt 2)\overline\eta>0\,.
\end{equation}
Hence for $(\gamma,\delta,\eta)$ sufficiently close to $(0,0,\overline{\eta})$ with $\gamma,\delta>0$, we get 
\begin{equation}\label{stima_frac}
\left\vert\frac{\tau/c}{\mu^++\mu^-}\right\vert^2\le\frac{\vert\tau/c\vert}{\Theta(\tau,\eta)}\,.
\end{equation}
We observe that the same estimate is true also for $\delta<0$ by noticing that
\[
\vert(\mu^++\mu^-)(\gamma,\delta,\eta)\vert^2=\vert(\mu^++\mu^-)(\gamma,-\delta,\eta)\vert^2
\]
(see \eqref{mod_somma}, \eqref{alphabeta}) and
\begin{equation*}
\left\vert\frac{\tau/c}{(\mu^++\mu^-)(\tau,\eta)}\right\vert^2=\left\vert\frac{\overline\tau/c}{(\mu^++\mu^-)(\overline\tau,\eta)}\right\vert^2\,.
\end{equation*}
From \eqref{lim_theta} and \eqref{stima_frac} we get
\begin{equation}\label{ext0}
\lim\limits_{(\tau,\eta)\to 0}\left(\frac{\tau/c}{\mu^++\mu^-}\right)^2=0\,,
\end{equation}
that is \eqref{extension} for $v=c$.

Consider now the case $\Re\tau=0$, that is $\tau=i\delta$ and, as above, assume that $\overline\eta>0$. If $\delta>0$, we may assume that $0<\frac{\delta}{c\eta}<2$ for $(\delta,\eta)$ sufficiently close to $(0,\overline\eta)$, being $\eta>0$; then from Lemma \ref{lemma_mu} (see formulas $\eqref{mu+}_4$ and $\eqref{mu-}_1$) we get
\begin{equation*}
\mu^+(i\delta,\eta)=i\sqrt{\left(\frac{\delta}{c}+\eta\right)^2-\eta^2}\,,\quad\mu^-(i\delta,\eta)=\sqrt{-\left(\frac{\delta}{c}-\eta\right)^2+\eta^2}\,.
\end{equation*}
If $\delta<0$ (so $-2<\frac{\delta}{c\eta}<0$ for $(\delta,\eta)$ close to $(0,\overline\eta)$) again from Lemma \ref{lemma_mu} (formulas $\eqref{mu+}_1$ and $\eqref{mu-}_3$) we get
\begin{equation*}
\mu^+(i\delta,\eta)=\sqrt{-\left(\frac{\delta}{c}+\eta\right)^2+\eta^2}\,,\quad\mu^-(i\delta,\eta)=-i\sqrt{\left(\frac{\delta}{c}-\eta\right)^2-\eta^2}\,.
\end{equation*}
From the above values of $\mu^\pm$ we get for all $\delta\neq 0$
\begin{equation}\label{mod_idelta}
\vert\mu^++\mu^-\vert^2=4\left\vert\frac{\delta}{c}\right\vert\eta\,,
\end{equation}
hence for $\tau=i\delta$
\[
\left\vert\frac{\tau/c}{\mu^++\mu^-}\right\vert^2=\frac{\vert\delta/c\vert}{4\eta}
\]
and passing to the limit as $(\delta,\eta)\to (0,\overline\eta)$ we obtain again \eqref{ext0}.

The same calculations can be repeated also in the case of $\overline{\eta}<0$.
\end{proof}

\begin{remark}Because of \eqref{extension}, the symbol $\Sigma$ can be extended to points $(0,\eta)$ where $\mu^++\mu^-$ vanishes. In particular we have the following limit for the coefficient in brackets, see \eqref{def_Sigma},
\begin{equation}\lim\limits_{(\tau,\eta)\to(0,\bar{\eta})}8\left( \frac{\tau/c}{\mu^++\mu^-} \right)^2-1=\frac{(v/c)^2-2}{(v/c)^2}  \, ,
\end{equation}
which changes sign according to $v/c\gtrless \sqrt2$.
This is in relation with the well-known stability criterion for vortex sheets, see \cite{CS04MR2095445,FM63MR0154509,M58MR0097930,S00MR1775057}; see also Remark \ref{ell_hyp}.
\begin{remark}\label{remark52}
We easily verify that $ \mu^++\mu^- $ is a homogeneous function of degree 1 in $(\tau,\eta)\in \Xi $ if $ \Re(\tau)>0. $ It follows that the continuous extension to points with $ \Re(\tau)=0$ of $ \frac{\tau/c}{\mu^++\mu^-} $ is homogeneous of degree 0 and the continuous extension of $ \Sigma $ is homogeneous of degree 2.
\end{remark}
	
\end{remark}
In the next lemma we study the roots of the symbol $\Sigma$.
\begin{lemma}\label{zeri_Sigma}
Let $ \Sigma(\tau,\eta) $ be the symbol defined in \eqref{def_Sigma}, for $ (\tau,\eta)\in\Xi. $
\begin{itemize}
\item[(i)] If $\frac{v}{c}<\sqrt{2}$, then $ \Sigma(\tau,\eta)=0 $ if and only if
\[
 \tau=cY_1|\eta|  \qquad \forall\eta\not=0\, ,  \]
where
\[ Y_1= \sqrt{-\left(\left(\frac{v}{c}\right)^2+1\right) + \sqrt{4\left(\frac{v}{c}\right)^2+1}}\,  .  \]

\item[(ii)] If $\frac{v}{c}>\sqrt{2}$, then $ \Sigma(\tau,\eta)=0 $ if and only if
\[ \tau=\pm icY_2\eta \qquad \forall\eta\not=0 \, ,  \]
where
\[ Y_2= \sqrt{\left(\frac{v}{c}\right)^2+1 - \sqrt{4\left(\frac{v}{c}\right)^2+1}}\,  .
\]
Each of these roots is simple. For instance, there exists a neighborhood $\V$ of $( icY_2\eta,\eta)$ in $\Xi_1$ and a $C^\infty$ function $H$ defined on $\V$ such that
\[ \Sigma(\tau,\eta)=(\tau-icY_2\eta)H(\tau,\eta), \quad H(\tau,\eta)\not=0 \quad\forall (\tau,\eta)\in\V.
 \]
 A similar result holds near $(-icY_2\eta,\eta)\in\Xi_1$.
\end{itemize}
\end{lemma}

\begin{remark}\label{ell_hyp}
(i) Recall that the equation \eqref{equ_f} was obtained by taking the Fourier transform with respect to $(t,x_1)$ of \eqref{puntoN7}, \eqref{wave3}, which corresponds to taking the Laplace transform with respect to $t$ and the Fourier transform with respect to $x_1$ of \eqref{puntoN6}, \eqref{wave2}. Taking the Fourier transform with respect to $t$ of \eqref{puntoN6}, \eqref{wave2} corresponds to the case $\gamma=\Re\tau=0 $, i.e. $ (\tau,\eta)=(i\delta,\eta) $.

If $\frac{v}{c}<\sqrt{2}$, from Lemma \ref{zeri_Sigma} the symbol $ \Sigma(\tau,\eta) $ only vanishes in points $ (\tau,\eta)$ with $\tau\in\R, \tau>0$. It follows that $ \Sigma(i\delta,\eta)\not=0 $ for all $(\delta,\eta)\in\R^2$. Therefore the symbol is elliptic, according to the standard definition. In this case planar vortex sheets are violently unstable, see \cite{S00MR1775057}.

(ii) If $\frac{v}{c}>\sqrt{2}$, $ \Sigma(\tau,\eta) $ vanishes in points $ (\tau,\eta)$ with $\Re\tau=0$, that is on the boundary of the frequency set $\Xi$. In this case planar vortex sheets are known to be weakly stable, in the sense that the so-called Lopatinski\u{\i} condition holds in a weak sense, see \cite{CS04MR2095445,FM63MR0154509,M58MR0097930,S00MR1775057}. For this case we expect a loss of derivatives for the solution with respect to the data.
\end{remark}

\begin{proof}[Proof of Lemma \ref{zeri_Sigma}]
 As we can easily verify $ \Sigma(\tau,0)=\tau^2\not=0 $ for $ (\tau,0)\in\Xi $ and  $ \Sigma(0,\eta)\neq 0 $ for $ (0,\eta)\in\Xi $ (see Corollary  \ref{zerosommamu} and Lemma \ref{extend}). Thus we assume without loss of generality that $\tau\neq 0$ and   $ \eta\not=0 $ and from Lemma \ref{diff_mu} $(\mu^+-\mu^-)(\tau,\eta)\neq 0$. We compute
\[
\frac{\tau/c}{\mu^++\mu^-}=\frac{(\tau/c)(\mu^+-\mu^-)}{(\mu^+)^2-(\mu^-)^2}=\frac{c}{4iv}\frac{\mu^+-\mu^-}{\eta}\,,
 \]
 \[
 \left(\frac{\mu^+-\mu^-}{\eta}\right)^2=
2\left(\left(\frac{\tau}{c\eta}\right)^2-\left(\frac{v}{c}\right)^2+1-\frac{\mu^+\mu^-}{\eta^2}\right)\,,
  \]
and substituting in \eqref{def_Sigma} gives
\begin{equation}\label{sigma1}
\Sigma=c^2\left(\mu^+\mu^--\eta^2\right).
\end{equation}
Let us introduce the quantities
\[
X:=\frac{\tau}{c\eta}, \qquad \tilde{\mu}^\pm:=\frac{\mu^\pm}{\eta}.
 \]
It follows from \eqref{sigma1} that
\begin{equation}\label{sigma2}
\Sigma=0 \qquad\mbox{if and only if }\quad \tilde{\mu}^+\tilde{\mu}^-=1 \,.
\end{equation}
Let us study the equation
\begin{equation}\label{muquadro}
(\tilde{\mu}^+)^2(\tilde{\mu}^-)^2=1 \,.
\end{equation}
This last equation is equivalent to the biquadratic equation
\[
X^4+2\left(\left(\frac{v}{c}\right)^2+1\right)X^2 + \left(\frac{v}{c}\right)^2 \left(\left(\frac{v}{c}\right)^2-2\right)=0 \,.
 \]
This is a polynomial equation of degree 2 in $ X^2 $ with real and distinct roots
\[
X^2=-\left(\left(\frac{v}{c}\right)^2+1\right)- \sqrt{4\left(\frac{v}{c}\right)^2+1} \,,
 \]
 and
\[
X^2=-\left(\left(\frac{v}{c}\right)^2+1\right)+ \sqrt{4\left(\frac{v}{c}\right)^2+1} \,.
\]
The first one gives the imaginary roots
\begin{equation}\label{roots_imag}
	X_{1,2}=\pm i Y_0, \qquad Y_0:=\sqrt{\left(\frac{v}{c}\right)^2+1 + \sqrt{4\left(\frac{v}{c}\right)^2+1}}\,.
\end{equation}
The second root of $ X^2 $ gives real or imaginary roots according to $ v/c \lessgtr \sqrt{2}.$ If $ v/c<\sqrt{2} $, there are 2 real and distinct roots
\begin{equation}\label{roots_real}
X_{3,4}=\pm Y_1, \qquad Y_1:= \sqrt{-\left(\left(\frac{v}{c}\right)^2+1\right) + \sqrt{4\left(\frac{v}{c}\right)^2+1}}\,.
\end{equation}
If $ v/c>\sqrt{2} $, there are 2 imaginary roots
\begin{equation}\label{roots_imag2}
X_{3,4}=\pm iY_2, \qquad Y_2:= \sqrt{\left(\frac{v}{c}\right)^2+1- \sqrt{4\left(\frac{v}{c}\right)^2+1}}\,.
\end{equation}
If $ v/c=\sqrt{2} $, then $ X_{3,4}=0 $.

Assume $ v/c<\sqrt{2} $ and consider first the real roots \eqref{roots_real}. In order to obtain $ \Re(\tau)>0 $ we choose $ X_3 $ or $ X_4 $ depending on $ \sgn(\eta), $
and we obtain that
$\tau=cY_1 |\eta|.$ By construction the pairs $ (cY_1 |\eta|,\eta) $ solve the equation \eqref{muquadro}. In order to verify if \eqref{sigma2} holds we proceed as follows. We compute
\[
(\tilde{\mu}^\pm)^2=a\pm ib, \qquad a=X^2-(v/c)^2+1, \qquad b=2Xv/c \,.
 \]
Recalling \eqref{roots} we can determine $ \tilde{\mu}^\pm $ and obtain that
\[
\tilde{\mu}^+ \tilde{\mu}^-=\left|\sqrt{\frac{r+a}{2}}+i\sqrt{\frac{r-a}{2}}\right|^2=r=|a+ib|>0\,.
\]
From \eqref{muquadro} it follows that $ \tilde{\mu}^+ \tilde{\mu}^-=1 $, that is \eqref{sigma2}.
Then, let us consider the imaginary roots in \eqref{roots_imag}. Correspondingly we have points $ (\tau,\eta)=(\pm icY_0\eta,\eta) $ with $ \Re(\tau)=0. $ Compairing the values $ \delta/(c\eta)=\pm Y_0$, where $Y_0>v/c+1,$ with the corresponding cases (i), (v) of Lemma \ref{super_mu} (if $ 1<v/c<\sqrt{2} $) and (i), (v)  of Lemma \ref{sub_mu} (if $ v/c<1 $), while in the sonic case $v=c$ we use Lemma \ref{lemma_mu}, we get
\[
\Re(\tilde{\mu}^+ \tilde{\mu}^-)=\eta^{-2}\Re(\mu^+\mu^-)<0\,.
 \]
It means that in such points $ \tilde{\mu}^+ \tilde{\mu}^-=-1 $, that is \eqref{sigma2} is not satisfied. Therefore we have proved that in case $ v/c<\sqrt{2} $,  the (only) roots of the symbol $\Sigma$ are the points $ (cY_1 |\eta|,\eta) $, for all $\eta\not=0$.

Now we assume $ v/c>\sqrt{2} $. For the imaginary roots \eqref{roots_imag} we can repeat the analysis made before.
Correspondingly to the roots $X_{1,2}$ we have the same points $ (\tau,\eta)=(\pm icY_0\eta,\eta) $ with $ \Re(\tau)=0. $ Compairing the value $ \delta/(c\eta)=\pm Y_0$, where $Y_0>v/c+1,$ with the different cases of Lemma \ref{super_mu} we get
\[
\Re(\tilde{\mu}^+ \tilde{\mu}^-)=\eta^{-2}\Re(\mu^+\mu^-)<0\,.
\]
It means that in such points $ \tilde{\mu}^+ \tilde{\mu}^-=-1 $, and \eqref{sigma2} is not satisfied.
Correspondingly to the roots $X_{3,4}$ in \eqref{roots_imag2} we have the points $ (\tau,\eta)=(\pm icY_2\eta,\eta) $ with $ \Re(\tau)=0. $ Because $ -(v/c-1)< \pm Y_0<v/c-1$, from Lemma \ref{super_mu} (iii) we deduce
\[
\Re(\tilde{\mu}^+ \tilde{\mu}^-)=\eta^{-2}\Re(\mu^+\mu^-)>0\,.
\]
It means that in such points $ \tilde{\mu}^+ \tilde{\mu}^-=1 $, and \eqref{sigma2} is satisfied. Therefore we have proved that in case $ v/c>\sqrt{2} $,  the (only) roots of the symbol $\Sigma$ are the points $ (\pm icY_2\eta,\eta) $, for all $\eta\not=0$.

This completes the first part of the proof of Lemma \ref{zeri_Sigma}; it remains to prove that the roots corresponding to $ X_{3,4}=\pm iY_2 $ are simple. Let us set $\sigma(X):=\tilde{\mu}^+ \tilde{\mu}^- -1$. From \eqref{sigma1} we have $\Sigma=c^2\eta^2\sigma(X)$. We wish to study $\sigma(X)$ in sufficiently small neighborhoods of points $ (\pm icY_2\eta,\eta) $. From Lemma \ref{lemma_mu} we may assume that $ \tilde{\mu}^\pm $ are different from 0 in such neighborhoods. First of all, from $$ (\tilde{\mu}^\pm)^2=(X\pm iv/c)^2 +1 $$ we obtain
\[
\frac{d\tilde{\mu}^+}{dX}=\frac{1}{\tilde\mu^+}(X+ iv/c), \qquad \frac{d\tilde{\mu}^-}{dX}=\frac{1}{\tilde\mu^-}(X- iv/c)\, .
 \]
We prove that
\begin{equation*}
\begin{array}{ll}
\ds \frac{d\sigma}{dX}(X)=\frac{d\tilde{\mu}^+}{dX}\tilde{\mu}^- + \tilde{\mu}^+ \frac{d\tilde{\mu}^-}{dX}=\frac{\tilde\mu^-}{\tilde\mu^+}(X+ iv/c) + \frac{\tilde\mu^+}{\tilde\mu^-}(X- iv/c)
\\
\ds=\frac{1}{\tilde\mu^+\tilde\mu^-}\left\{\left((\tilde\mu^+)^2+(\tilde\mu^-)^2\right)X -i\left((\tilde\mu^+)^2-(\tilde\mu^-)^2\right) v/c\right\}
\\
\ds=\frac{2X}{\tilde\mu^+\tilde\mu^-}\left\{X^2+(v/c)^2+1\right\}
\,.
\end{array}
\end{equation*}
Moreover we have
\begin{equation*}
\sigma(X_{3})=0, \qquad K:=\left\{X^2+(v/c)^2+1\right\}_{|X=X_{3}}>0 \,.
\end{equation*}
Consequently we can write
\[
\sigma(X)=(X-X_3)\, \tilde{H}(X)\,,
 \]
where, by continuity $ \tilde{H}(X)\not=0 $ in a neighborhood of $ X=X_3, $ because $ \tilde{H}(X_3)=\frac{d\sigma}{dX}(X_3)=2X_3K\not=0 $.
Thus we write
\begin{equation*}
\Sigma(\tau,\eta)=c^2\eta^2\sigma(X)=c^2\eta^2\sigma\left(\frac{\tau}{c\eta}\right)=(\tau-X_3c\eta)\,H(\tau,\eta), \qquad H(\tau,\eta):=c\eta\, \tilde{H}\left(\frac{\tau}{c\eta}\right)\,.
\end{equation*}
Since
\[
H(X_3c\eta,\eta)=c\eta\, \tilde{H}\left(X_3\right)\not=0 \qquad\forall \eta\not=0\,,
 \]
by continuity $ H(\tau,\eta)\not=0 $ in a small neighborhood of $ (X_3c\eta,\eta) $. It is easily verified that $H$ is a homogeneous function of degree 1. By the same argument we prove the similar result for $ X=X_4. $ The proof of Lemma \ref{zeri_Sigma} is complete.
\end{proof}

\section{Proof of Theorem \ref{teoexist}}

\begin{lemma}
Let $\Sigma$ be the symbol defined by \eqref{def_Sigma} and $s\in\R, \gamma\ge1 $. Given any $f\in H^{s+2}_\gamma(\R^2)$, let $g$ be the function defined by
\begin{equation}\label{def_g}
\Sigma(\tau,\eta)\widehat f(\tau,\eta)=\widehat g(\tau,\eta) \qquad (\tau,\eta)\in\Xi\, ,
\end{equation}
where $ \widehat{g} $ is the Fourier transform of $\widetilde{g}:= e^{-\gamma t}g. $ Then $g\in H^{s}_\gamma(\R^2)$ with
\begin{equation*}
\|g\|_{H^{s}_\gamma(\R^2)}\le C \|f\|_{H^{s+2}_\gamma(\R^2)} \, ,
\end{equation*}
for a suitable positive constant $C$ independent of $ \gamma $.
\end{lemma}
\begin{proof}
The proof follows by observing that $ \Sigma(\tau,\eta)$ is a homogeneous function of degree 2 on $\Xi$, so there exists a positive constant $C$ such that
\begin{equation}\label{stimaSigma}
|\Sigma(\tau,\eta)|\le C(|\tau|^2+\eta^2)=C\Lambda^2(\tau,\eta) \qquad \forall (\tau,\eta)\in\Xi\,.
\end{equation}
Then
\begin{equation*}
\|g\|_{H^{s}_\gamma(\R^2)}=\frac{1}{2\pi}\|\Lambda^s\widehat{g}\|=\frac{1}{2\pi}\|\Lambda^s\Sigma\widehat{f}\|
\le C\|\Lambda^{s+2}\widehat{f}\|=C \|f\|_{H^{s+2}_\gamma(\R^2)} \, .
\end{equation*}
\end{proof}
In the following theorem we prove the a priori estimate of the solution $f$ to equation \eqref{def_g}, for a given $g$.
\begin{theorem}\label{teofg}
Assume $\frac{v}{c}>\sqrt{2}$. Let $\Sigma$ be the symbol defined by \eqref{def_Sigma} and $s\in\R$. Given any $f\in H^{s+2}_\gamma(\R^2)$, let $g\in H^s_\gamma(\R^2)$ be the function defined by \eqref{def_g}.
Then there exists a positive constant $C$ such that for all $\gamma\ge1$ the following estimate holds
\begin{equation}\label{stimafg}
\gamma \|f\|_{H^{s+1}_\gamma(\R^2)} \le C \|g\|_{H^s_\gamma(\R^2)} \, .
\end{equation}
\end{theorem}
\begin{proof}

The study of $\Sigma$ in the proof of Lemma \ref{zeri_Sigma} implies that for all $(\tau_0,\eta_0)\in\Xi_1$, there exists a neighborhood $\V$ of $(\tau_0,\eta_0)$ with suitable properties, as explained in the following. Because $\Xi_1$ is a $C^\infty$ compact manifold, there exists a finite covering $(\V_1,\dots,\V_I)$ of $\Xi_1$ by such neighborhoods, and a smooth partition of unity $(\chi_1,\dots,\chi_I)$ associated with this covering.
The $\chi_i's$ are nonnegative $C^\infty$ functions with
	\[
	\supp\chi_i\subset\V_i, \qquad \sum_{i=1}^I\chi_i^2=1.
	\]
	We consider two different cases.
	
{\it In the first case}
$\V_i$ is a neighborhood of an {\it elliptic} point, that is a point $ (\tau_0,\eta_0) $ where $ \Sigma(\tau_0,\eta_0)\not=0. $ By taking $\V_i$ sufficiently small we may assume that $ \Sigma(\tau,\eta)\not=0 $ in the whole neighborhood $\V_i$, and there exists a positive constant $C$ such that $$ |\Sigma(\tau,\eta)|\ge C \qquad \forall (\tau,\eta)\in\V_i\,.
$$
Let us extend the associated function $\chi_i$ to the whole set of frequencies $\Xi$, as a homogeneous mapping of degree 0 with respect to $(\tau,\eta)$.
$ \Sigma(\tau,\eta)$ is a homogeneous function of degree 2 on $\Xi$, so we have
\begin{equation}\label{stima_ellip0}
|\Sigma(\tau,\eta)|\ge C(|\tau|^2+\eta^2) \qquad \forall (\tau,\eta)\in\V_i\cdot\R^+\,.
\end{equation}
We deduce that
\begin{equation}\label{stima_ellip}
C(|\tau|^2+\eta^2)|\chi_i\widehat f(\tau,\eta)|\le
|\Sigma(\tau,\eta)\chi_i\widehat f(\tau,\eta)|=|\chi_i\widehat g(\tau,\eta)|  \qquad \forall (\tau,\eta)\in\V_i\cdot\R^+\,.
\end{equation}

{\it In the second case}\label{first}
$\V_i$ is a neighborhood of a {\it root} of the symbol $ \Sigma $, i.e. a point $ (\tau_0,\eta_0) $ where $ \Sigma(\tau_0,\eta_0)=0. $ For instance we may assume that $ (\tau_0,\eta_0)=(icY_2\eta_0,\eta_0), \eta_0\not=0$, see Lemma \ref{zeri_Sigma}; a similar argument applies for the other family of roots $(\tau,\eta)=(-icY_2\eta,\eta)$.
According to Lemma \ref{zeri_Sigma} we may assume that on $\V_i$ it holds
\[ \Sigma(\tau,\eta)=(\tau-icY_2\eta)H(\tau,\eta), \quad H(\tau,\eta)\not=0 \quad\forall (\tau,\eta)\in\V_i.
\]
We extend the associated function $\chi_i$ to the whole set of frequencies $\Xi$, as a homogeneous mapping of degree 0 with respect to $(\tau,\eta)$.
Because $ H(\tau,\eta)\not=0 $ on $\V_i$, there exists a positive constant $C$ such that $$ |H(\tau,\eta)|\ge C \qquad \forall (\tau,\eta)\in\V_i\,.
$$
$H(\tau,\eta)$ is a homogeneous function of degree 1 on $\Xi$, so we have
\[
|H(\tau,\eta)|\ge C(|\tau|^2+\eta^2)^{1/2} \qquad \forall (\tau,\eta)\in\V_i\cdot\R^+\,.
\]	
Then we obtain
\begin{equation}\label{stima_nellip0}
|\Sigma(\tau,\eta)|=|(\tau-icY_2\eta)H(\tau,\eta)|\ge C\gamma(|\tau|^2+\eta^2)^{1/2} \qquad \forall (\tau,\eta)\in\V_i\cdot\R^+\,,
\end{equation}
and we deduce that
\begin{equation}\label{stima_nellip}
C\gamma(|\tau|^2+\eta^2)^{1/2}|\chi_i\widehat f(\tau,\eta)|\le
|\chi_i\widehat g(\tau,\eta)|  \qquad \forall (\tau,\eta)\in\V_i\cdot\R^+\,.
\end{equation}

In conclusion, adding up the square of \eqref{stima_ellip} and \eqref{stima_nellip}, and using that the $ \chi_i $'s form a partition of unity gives
\begin{equation}\label{stimapointwise}
C\gamma^2(|\tau|^2+\eta^2)|\widehat f(\tau,\eta)|^2\le
|\widehat g(\tau,\eta)|^2  \qquad \forall (\tau,\eta)\in\Xi\,.
\end{equation}
Multiplying the previous inequality by $ (|\tau|^2+\eta^2)^s $, integrating with respect to $ (\delta,\eta)\in\R^2 $ and using Plancherel's theorem finally yields the estimate
\begin{equation*}
\gamma^2\|\widetilde{f}\|_{s+1,\gamma}^2\le C\|\widetilde{g}\|_{s,\gamma}^2\, ,
\end{equation*}
for a suitable constant $C$, that is  \eqref{stimafg}.
\end{proof}

In the following theorem we prove the existence of the solution $f$ to equation \eqref{def_g}.

\begin{theorem}\label{teoexistfg}
Assume $\frac{v}{c}>\sqrt{2}$. Let $\Sigma$ be the symbol defined by \eqref{def_Sigma} and $s\in\R, \gamma\ge1$. Given any $g\in H^s_\gamma(\R^2)$ there exists a unique solution $f\in H^{s+1}_\gamma(\R^2)$ of equation \eqref{def_g}, satisfying the estimate
\eqref{stimafg}.
\end{theorem}
\begin{proof}
We use a duality argument. Let us denote by $ \Sigma^* $ the symbol of the adjoint of the operator with symbol $ \Sigma $, such that
\begin{align*}
\langle \Sigma\widehat{f},\widehat{h}\rangle=\langle \widehat{f},\Sigma^*\widehat{h}\rangle
\end{align*}
for $ f,h $ sufficiently smooth. From the definition \eqref{def_Sigma} we easily deduce that
\begin{equation}\label{equSS*}
\Sigma^\ast(\tau,\eta)=\Sigma(\bar{\tau},\eta)\,.
\end{equation}
Thus, from Theorem \ref{teofg}, see in particular \eqref{stima_ellip0}, \eqref{stima_nellip0}, \eqref{stimapointwise}, we obtain the estimate
 \begin{equation*}
\gamma^2(|\tau|^2+\eta^2)|\widehat h(\tau,\eta)|^2\le
C|\Sigma^\ast(\tau,\eta)\widehat h(\tau,\eta)|^2  \,,
 \end{equation*}
 which gives by integration in $ (\delta,\eta) $
\begin{equation}\label{stimaSigma*}
\gamma\|\Lambda\widehat h \|\le
C\|\Sigma^\ast\widehat h\|  \,.
\end{equation}
We compute
\begin{align}\label{duality}
\left|\langle \widehat{g},\widehat{h}\rangle\right|=\left| \langle \Lambda^s\widehat{g},\Lambda^{-s}\widehat{h}\rangle\right|
\le\|\Lambda^s\widehat{g}\| \, \|\Lambda^{-s}\widehat{h}\|\,.
\end{align}
From \eqref{stimaSigma}, \eqref{equSS*}, \eqref{stimaSigma*} (with $\Lambda^{-s-1}\widehat{h}$ instead of $\widehat{h}$) we obtain
\begin{equation}\label{stimaL-s}
\|\Lambda^{-s}\widehat{h}\|=\|\Lambda\Lambda^{-s-1}\widehat{h}\|\le \frac{C}{\gamma}\|\Sigma^\ast\Lambda^{-s-1}\widehat h\| \le \frac{C}{\gamma}\|\Lambda^{-s+1}\widehat h\|= \frac{C}{\gamma}\|h\|_{H^{-s+1}_\gamma(\R^2)} \, .
\end{equation}
Let us denote
\[
\Rc:=\left\{ \Sigma^\ast\Lambda^{-s-1}\widehat h \,\, | \,\, h\in H^{-s+1}_\gamma(\R^2) \right\} \,.
 \]
 From \eqref{stimaL-s} it is clear that $ \Rc$ is a subspace of $ L^2(\R^2) $; moreover, the map $ \Sigma^\ast\Lambda^{-s-1}\widehat h \mapsto \Lambda^{-s}\widehat{h} $ is well-defined and continuous from $ \Rc$ into $ L^2(\R^2) $. Given $ g\in H^s_\gamma(\R^2) $, we define a linear form $ \ell $ on $ \Rc $ by
 \[
 \ell(\Sigma^\ast\Lambda^{-s-1}\widehat h)= \langle \widehat{g},\widehat{h} \rangle \,.
  \]
From \eqref{duality}, \eqref{stimaL-s} we obtain
\[
\left| \ell(\Sigma^\ast\Lambda^{-s-1}\widehat h)\right| \le \frac{C}{\gamma}\|\Lambda^s\widehat{g}\|  \, \|\Sigma^\ast\Lambda^{-s-1}\widehat h\| \,.
 \]
 Thanks to the Hahn-Banach and Riesz theorems, there exists a unique $ w\in L^2(\R^2) $ such that
 \[
\langle w,\Sigma^\ast\Lambda^{-s-1}\widehat h \rangle = \ell(\Sigma^\ast\Lambda^{-s-1}\widehat h)\,, \qquad
 \|w\|= \|\ell\|_{\Lc(\Rc)} \le \frac{C}{\gamma}\|\Lambda^s\widehat{g}\| \,.
  \]
Defining $ \widehat f:= \Lambda^{-s-1}w$ we get $ f\in H^{s+1}_\gamma(\R^2) $ such that
\[
\langle \Sigma\widehat f, \widehat h \rangle = \langle \widehat f, \Sigma^\ast\widehat h \rangle= \langle \widehat{g},\widehat{h} \rangle \qquad \forall h\in H^{-s+1}_\gamma(\R^2)\,,
 \]
which shows that $ f $ is a solution of equation \eqref{def_g}. Moreover
\[
\|f\|_{H^{s+1}_\gamma(\R^2)}=\frac{1}{2\pi}\|\Lambda^{s+1}\widehat{f}\|=\frac{1}{2\pi}\|w\| \le \frac{C}{\gamma}\|\Lambda^s\widehat{g}\|=\frac{C}{\gamma}\|g\|_{H^{s}_\gamma(\R^2)}\,,
 \]
 that is  \eqref{stimafg}. The uniqueness of the solution follows from the linearity of the problem and the a priori estimate.
\end{proof}

Now we can conclude the proof of Theorem \ref{teoexist}.
\begin{proof}[Proof of Theorem \ref{teoexist}]
We apply the result of Theorem \ref{teoexistfg} for \[ \widehat g(\tau,\eta)=-\frac{\mu^+\mu^-}{\mu^++\mu^-}\,M \, , \]
with $M$ defined in \eqref{def_M}. We write
\begin{equation*}
\widehat g=\widehat g_1-\widehat g_2,
\end{equation*}
where
\begin{equation*}
\widehat g_1=-\frac{\mu^-}{\mu^++\mu^-}\int_{0}^{+\infty}e^{-\mu^+ y}\widehat{\F}^+ (\cdot, y)\, \d y \,, \qquad
\widehat g_2=-\frac{\mu^+}{\mu^++\mu^-}\int_{0}^{+\infty}e^{-\mu^- y}\widehat{\F}^- (\cdot,- y)\, \d y \, .
\end{equation*}
By the Plancherel theorem and Cauchy-Schwarz inequality we have
\begin{equation}\label{stimag1}
\begin{array}{ll}
\ds \|g_1\|_{H^s_\gamma(\R^2)}^2=\frac{1}{(2\pi)^2}\iint_{\R^2} \Lambda^{2s}\left| \frac{\mu^-}{\mu^++\mu^-}\int_{0}^{+\infty}e^{-\mu^+ y}\widehat{\F}^+ (\cdot, y)\, dy \right|^2\d\delta \d\eta\\
\ds \le \frac{1}{(2\pi)^2}\iint_{\R^2} \Lambda^{2s}\left| \frac{\mu^-}{\mu^++\mu^-}\right|^2\frac{1}{2\Re\mu^+}
\left(\int_{0}^{+\infty}|\widehat{\F}^+ (\cdot, y)|^2\, dy\right) \d\delta \d\eta \, .
\end{array}
\end{equation}
Then we use the fact that $\frac{\mu^-}{\mu^++\mu^-}$ is a homogeneous function of degree zero in $\Xi$ so that
\[ \left| \frac{\mu^-}{\mu^++\mu^-}\right|^2\le C \qquad \forall (\tau,\eta)\in\Xi\, ,
\]
for a suitable constant $C>0$. Moreover, we have the estimate from below
\begin{equation*}\Re\mu^+\ge \frac{1}{\sqrt{2}\,c}\,\gamma\, ,
\end{equation*}
see Lemma \ref{stima_Re_mu}. Thus we obtain from \eqref{stimag1}
\begin{equation*}
\begin{array}{ll}
\ds \|g_1\|_{H^s_\gamma(\R^2)}^2
 \le \frac{C}{\gamma}\iint_{\R^2} \Lambda^{2s}
\left(\int_{0}^{+\infty}|\widehat{\F}^+ (\cdot, y)|^2\, \d y\right) \d\delta \d\eta =\frac{C}{\gamma}\|\F^+\|_{L^2(\R^+;H^s_\gamma(\R^2))}^2 \, .
\end{array}
\end{equation*}
The proof of the estimate of $g_2$ is similar. This completes the proof of Theorem \ref{teoexist}.
\end{proof}

\subsection*{Acknowledgement}
The research was supported in part by the Italian research
project PRIN 2015 ``Hyperbolic Systems of Conservation Laws and Fluid Dynamics: Analysis and Applications''.




\end{document}